\documentclass[
final
]{dmtcs-episciences}
\usepackage[utf8]{inputenc}
\usepackage{subfigure}

\usepackage{mathtools}
\usepackage{amsmath}
\usepackage{amsthm}
\usepackage{amssymb}
\usepackage{microtype}
\usepackage[shortlabels]{enumitem}
\makeatletter
\numberwithin{equation}{section}
\numberwithin{figure}{section}
\usepackage{bm}
\usepackage{multicol}

\usepackage[dvipsnames]{xcolor}
\usepackage{tikz}
\usetikzlibrary{arrows}

\usetikzlibrary{shapes}

\usepackage{algorithm}
\usepackage{algpseudocode}
\tikzstyle{pathdefault}=[draw, line width=1, solid, color=black]
\tikzstyle{nodedefault}=[circle, inner sep=1.5, fill=black]
\tikzstyle{empty}=[]
\tikzstyle{nodeellipsis}=[circle, inner sep=0.5, fill=black]
\tikzstyle{pathcolor1}=[draw, line width=1.3, densely dashed, color=red]
\tikzstyle{pathcolor2}=[draw, line width=1.6, densely dotted, color=blue]
\tikzstyle{pathcolorlight}=[draw, line width=1, dotted, color=lightgray]

\tikzstyle{arbpathcolor0}=[line width=1, dashdotted, color=black]
\tikzstyle{arbpathcolor1}=[line width=1, densely dashed, color=red]
\tikzstyle{arbpathdefault}=[line width=1, densely dotted, color=blue]

\newcounter{id}
\newcommand{\drawlinedotswithstyle}[4]{
 \def\x{{#3}}
 \def\y{{#4}}
 \tikzstyle{thispathstyle}=[#1]
 \tikzstyle{thisnodestyle}=[#2]
 \setcounter{id}{-1} %start id at -1
 \foreach \j in {#3}{\stepcounter{id}} %id is one less than num of pts
 \foreach \i in {1,...,\the\value{id}}{  %loop through later indices
  \path[thispathstyle] (\x[\i],\y[\i]) --(\x[\i-1],\y[\i-1]); %draw edge
 }
 \foreach \i in {1,...,\the\value{id}}{  %loop through later indices
  \node[thisnodestyle] at (\x[\i],\y[\i]) {}; %draw node
 }
 \node[thisnodestyle] at (\x[0],\y[0]) {}; %draw first node outside of loop
}

\DeclareDocumentCommand{\drawlinedots}{ O{pathdefault} O{nodedefault} m m}{\drawlinedotswithstyle{#1}{#2}{#3}{#4}}

\definecolor{dartmouthgreen}{HTML}{00693E} %official Dartmouth College color
\definecolor{davidsonred}{HTML}{AC1A2F} %official Davidson College color

\usepackage{hyperref}
\usepackage{colortbl}
\usepackage{xparse}
\usepackage[font=small]{caption}

\DeclareMathOperator{\st}{st}
\DeclareMathOperator{\des}{des}
\DeclareMathOperator{\rix}{rix}
\DeclareMathOperator{\exc}{exc}
\DeclareMathOperator{\fix}{fix}
\DeclareMathOperator{\Rix}{Rix}
\DeclareMathOperator{\maj}{maj}
\DeclareMathOperator{\ai}{ai}
\DeclareMathOperator{\pk}{pk}
\DeclareMathOperator{\dbl}{dbl}
\DeclareMathOperator{\Fix}{Fix}

\makeatletter
\theoremstyle{plain}
\newtheorem{thm}{\protect\theoremname}
\theoremstyle{plain}
\newtheorem{lem}[thm]{\protect\lemmaname}
\theoremstyle{plain}
\newtheorem{cor}[thm]{\protect\corollaryname}
\theoremstyle{plain}
\newtheorem{prop}[thm]{\protect\propositionname}
\theoremstyle{definition}
\newtheorem{defn}[thm]{Definition}
\theoremstyle{definition}
\newtheorem{algorithm_env}[thm]{Algorithm}
\theoremstyle{definition}
\newtheorem{example}[thm]{Example}
\theoremstyle{definition}
\newtheorem{remark}[thm]{Remark}

\makeatother

\usepackage{titlesec}
\titleformat{\section}
  {\normalfont\large\bfseries}{\thesection.}{.8ex}{} %changed from Large
\titleformat{\subsection}
  {\normalfont\bfseries}{\thesubsection.}{.8ex}{} %large removed
\let\originalleft\left
\let\originalright\right
\renewcommand{\left}{\mathopen{}\mathclose\bgroup\originalleft}
\renewcommand{\right}{\aftergroup\egroup\originalright}

\providecommand{\corollaryname}{Corollary}
\providecommand{\lemmaname}{Lemma}
\providecommand{\theoremname}{Theorem}
\providecommand{\propositionname}{Proposition}

\author[Nadia Lafreni\`{e}re and Yan Zhuang]{Nadia Lafreni\`{e}re\affiliationmark{1}
  \and Yan Zhuang\affiliationmark{2}\thanks{YZ was partially supported by an AMS-Simons Travel Grant and NSF grant DMS-2316181.}}
\title[On the $\rix$ statistic and valley-hopping]{On the $\rix$ statistic and valley-hopping}
\affiliation{
  Department of Mathematics, Dartmouth College, Hanover, NH, USA\\
  Department of Mathematics and Computer Science, Davidson College, Davidson, NC, USA}
\makeatother

\begin{document}

\publicationdata{vol. 26:2}{2024}{3}{10.46298/dmtcs.11553}{2023-07-07; None}{2024-01-05}

\maketitle

\begin{abstract}
This paper studies the relationship between the modified Foata--Strehl action (a.k.a.\ valley-hopping)---a group action on permutations used to demonstrate the $\gamma$-positivity of the Eulerian polynomials---and the number of rixed points $\rix$---a recursively-defined permutation statistic introduced by Lin in the context of an equidistribution problem. We give a linear-time iterative algorithm for computing the set of rixed points, and prove that the $\rix$ statistic is homomesic under valley-hopping. We also demonstrate that a bijection $\Phi$ introduced by Lin and Zeng in the study of the $\rix$ statistic sends orbits of the valley-hopping action to orbits of a cyclic version of valley-hopping, which implies that the number of fixed points $\fix$ is homomesic under cyclic valley-hopping.
\end{abstract}

{\let\thefootnote\relax}

\section{Introduction}
Let $\mathfrak{S}_{n}$ denote the symmetric group of permutations of $[n]\coloneqq\{1,2,\dots,n\}$. We will usually write permutations in one-line notation, so $\pi\in\mathfrak{S}_{n}$ is written as $\pi=\pi_{1}\pi_{2}\cdots\pi_{n}$. Given a permutation $\pi\in\mathfrak{S}_{n}$, we call $\pi_i$ (where $i\in[n-1]$) a \textit{descent} of $\pi$ if $\pi_i>\pi_{i+1}$ and we call $i\in[n]$ an \textit{excedance} of $\pi$ if $i<\pi_i$. Let $\des(\pi)$ denote the number of descents of $\pi$, and $\exc(\pi)$ its number of excedances. 

The descent number $\des$ and excedance number $\exc$ are classical permutation statistics which are well known to be equidistributed over $\mathfrak{S}_{n}$; in other words, for any fixed $n$ and $k$, the number of permutations in $\mathfrak{S}_{n}$ with $\des(\pi)=k$ is equal to the number of permutations in $\mathfrak{S}_{n}$ with $\exc(\pi)=k$. The distributions of both statistics are encoded by the \textit{Eulerian polynomials} 
\[
A_{n}(t)\coloneqq\sum_{\pi\in\mathfrak{S}_{n}}t^{\des(\pi)}=\sum_{\pi\in\mathfrak{S}_{n}}t^{\exc(\pi)}.
\]
The Eulerian polynomials are \textit{$\gamma$-positive}: there exist non-negative coefficients $\gamma_{k}$ for which 
\[
A_{n}(t)=\sum_{k=0}^{\left\lfloor (n-1)/2\right\rfloor }\gamma_{k}t^{k}(1+t)^{n-1-2k}.
\]
The $\gamma$-positivity of $A_{n}(t)$ implies that the coefficients of $A_{n}(t)$ are unimodal and symmetric. One method of proving the $\gamma$-positivity of $A_{n}(t)$ which yields a combinatorial interpretation of the coefficients $\gamma_{k}$ is via a group action on permutations which we will refer to as \textit{valley-hopping} (but which is also called the \textit{modified Foata--Strehl action} in the literature). One can show that the distribution of $\des$ over each orbit of the valley-hopping action is of the form $t^{k}(1+t)^{n-1-2k}$, and summing over all orbits completes the proof. Valley-hopping and its variants have been used to provide combinatorial proofs for related $\gamma$-positivity results, some of which are described in the survey article \cite{Athanasiadis[20162018]} on $\gamma$-positivity in combinatorics and geometry.

This paper concerns the relationship between valley-hopping and a permutation statistic denoted $\rix$. The $\rix$ statistic is defined recursively in the following way: let $\rix(\varnothing)=0$, and if $w = w_1 w_2 \cdots w_k$ is a word with distinct positive integer letters and largest letter $w_i$, then let
\begin{equation}
\rix(w)\coloneqq\begin{cases}
0, & \text{if }i=1<k,\\
1+\rix(w_{1}w_{2}\cdots w_{k-1}), & \text{if }i=k,\\
\rix(w_{i+1} w_{i+2}\cdots w_{k}), & \text{if }1<i<k.
\end{cases} \label{e-rix}
\end{equation}
Equivalently, $\rix(\pi)$ is the number of ``rixed points'' of $\pi$ as defined by Lin and Zeng \cite{Lin2015}.

The $\rix$ statistic was introduced by Lin \cite{Lin2013} in the context of an equidistribution problem which we will now describe. Consider the \textit{basic Eulerian polynomials} 
\[
A_{n}(t,q,r)\coloneqq\sum_{\pi\in\mathfrak{S}_{n}}t^{\exc(\pi)}r^{\fix(\pi)}q^{\maj(\pi)-\exc(\pi)}
\]
where $\fix(\pi)$ is the number of fixed points of $\pi$ and $\maj(\pi)\coloneqq\sum_{\pi_{i}>\pi_{i+1}}i$ is the \textit{major index} of $\pi$ (the sum of its descents). Since $\des$ and $\exc$ are equidistributed over $\mathfrak{S}_{n}$, one may ask if there are statistics $\st_{1}$ and $\st_{2}$ for which $(\exc,\fix,\maj-\exc)$ and $(\des,\st_{1},\st_{2})$ are equidistributed over $\mathfrak{S}_{n}$. Lin showed that we can take $\st_{1}=\rix$ and $\st_{2}=\ai$, where the latter is the number of ``admissible inversions'' of $\pi$. In other words, 
\[
A_{n}(t,q,r)=\sum_{\pi\in\mathfrak{S}_{n}}t^{\des(\pi)}r^{\rix(\pi)}q^{\ai(\pi)}
\]
is an alternative interpretation for the basic Eulerian polynomials.

Prior to Lin's introduction of the $\rix$ statistic, Shareshian and Wachs \cite{Shareshian2010} had remarked that the polynomials $A_{n}(t,0,r)$ and $A_{n}(t,1,r)$ satisfy a refined version of $\gamma$-positivity. Later, Lin and Zeng \cite{Lin2015} used valley-hopping and Lin's interpretation of the basic Eulerian polynomials to give combinatorial interpretations for the $\gamma$-coefficients of $A_{n}(t,0,r)$ and $A_{n}(t,1,r)$. Along the way, they defined a bijection $\Phi\colon\mathfrak{S}_{n}\rightarrow\mathfrak{S}_{n}$ satisfying $\des(\pi)=\exc(\Phi(\pi))$ and $\Rix(\pi)=\Fix(\Phi(\pi))$ where $\Rix$ is the set of rixed points and $\Fix$ the set of fixed points. The existence of this bijection $\Phi$ demonstrates that not only are $(\exc,\fix)$ and $(\des,\rix)$ jointly equidistributed over $\mathfrak{S}_{n}$, but $(\exc,\Fix)$ and $(\des,\Rix)$ are as well.\footnote{The same is not true for $(\exc,\Fix,\maj-\exc)$ and $(\des,\Rix,\ai)$.}

Dynamical algebraic combinatorics\textemdash which, broadly speaking, investigates phenomena associated with actions on combinatorial structures\textemdash is an emerging area of research within algebraic combinatorics. An example of such phenomena is \textit{homomesy}, where a statistic on a set of combinatorial objects has the same average value over each orbit of an action; see \cite{Roby2016} for a survey of this topic. Motivated by \cite{Elder2022}, which was a systematic investigation of the homomesy phenomenon on permutations, the present work originated as an attempt to identify permutation statistics which are homomesic under valley-hopping, which was not considered in \cite{Elder2022}. Following the approach of \cite{Elder2022}, we automatically searched for permutation statistics from the online FindStat database \cite{FindStat} that exhibited homomesic behavior under valley-hopping for $2\leq n \leq 6$. Our positive matches included the descent number and some related statistics (such as the number of ascents), but aside from these, the only statistic that appeared to be homomesic under valley-hopping is the $\rix$ statistic.\footnote{As of June 28, 2023, there were 387 other permutation statistics with code in the FindStat database, but we found counterexamples for all of them.}

A simple examination of the orbit structure of the valley-hopping action shows that $\des$ is homomesic under valley-hopping,\footnote{In fact, the homomesy of $\des$ under valley-hopping is implicit in the valley-hopping proof for the $\gamma$-positivity of the Eulerian polynomials.} which in turn implies that the related statistics are homomesic by way of symmetry arguments. On the other hand, proving that $\rix$ is homomesic under valley-hopping required further investigation, and it soon became evident to us that there is more to the relationship between the $\rix$ statistic and valley-hopping than meets the eye. Our subsequent explorations on this topic led to the full results presented here.

The organization of our paper is as follows. Section 2 introduces the definitions of several permutation statistics that are relevant to this work, as well as Lin and Zeng's ``rix-factorization'' of a permutation (which is needed to define the $\Rix$ statistic). Section 3 is devoted to a linear-time iterative algorithm for computing the set of rixed points and the rix-factorization. After we present and demonstrate the validity of our algorithm, we will use this algorithm to help prove several results, including a recursive definition for $\Rix$ that lifts the definition of $\rix$ given in \eqref{e-rix}, another characterization for rixed points, and a few additional lemmas about rixed points and the rix-factorization which we will use later on in this paper.

Section 4 is again expository, and provides the definition of valley-hopping and introduces a cyclic version of valley-hopping. Cyclic valley-hopping was originally defined on derangements by Sun and Wang \cite{Sun2014}, and was extended to all permutations by Cooper, Jones, and the second author~\cite{Cooper2020}. While the version of cyclic valley-hopping due to Cooper--Jones--Zhuang fixes all fixed points, our version of cyclic valley-hopping does not fix fixed points. We also define ``restricted'' versions of valley-hopping and cyclic valley-hopping. Restricted valley-hopping was first introduced by Lin and Zeng \cite{Lin2015}, whereas restricted cyclic valley-hopping is precisely the version of cyclic valley-hopping due to Cooper--Jones--Zhuang mentioned above.

Sections 5--6 focus on our main results concerning the relationship between valley-hopping and the $\rix$ statistic. In Section 5, we prove that $\rix$ is homomesic under valley-hopping. Finally, in Section 6, we show that the bijection $\Phi$ of Lin and Zeng sends valley-hopping orbits to cyclic valley-hopping orbits (and sends restricted valley-hopping orbits to restricted cyclic valley-hopping orbits). As a consequence, $\fix$ (the number of fixed points) is homomesic under cyclic valley-hopping.

\section{Permutation statistics}

The purpose of this preliminary section is to introduce several permutation statistics that will be relevant to our work. Fix a permutation $\pi=\pi_1\pi_2\cdots\pi_n$ in $\mathfrak{S}_n$. We have already defined descents; an \textit{ascent} of $\pi$ is a letter $\pi_i$ (where $i\in[n]$) for which $\pi_i < \pi_{i+1}$, with the convention $\pi_{n+1}=\infty$---i.e., an ascent is a letter that is not a descent. For example, take $\pi = 135987426$. Then the ascents of $\pi$ are $1$, $3$, $5$, $2$, and $6$, whereas its descents are $\pi$ are $9$, $8$, $7$, and $4$. Notice that, under our definition, the last letter of a permutation is always an ascent.

Let us adopt the convention $\pi_{0}=\pi_{n+1}=\infty$ for the definitions below. Given $i\in[n]$, we call $\pi_i$:
\begin{itemize}
\item a \textit{peak} of $\pi$ if $\pi_{i-1}<\pi_i>\pi_{i+1}$;
\item a \textit{valley} of $\pi$ if $\pi_{i-1}>\pi_i<\pi_{i+1}$;
\item a \textit{double ascent} of $\pi$ if $\pi_{i-1}<\pi_i<\pi_{i+1}$;
\item a \textit{double descent} of $\pi$ if $\pi_{i-1}>\pi_i>\pi_{i+1}$.
\end{itemize}
Continuing the example above, the only peak of $\pi = 135987426$ is $9$; its valleys are $1$ and $2$; its double ascents are $3$, $5$, and $6$; and its double descents are $8$, $7$, and $4$. In particular, observe that every letter of a permutation is either a peak, valley, double ascent, or double descent.

We note that the terms ascent, descent, peak, valley, double ascent, and double descent more commonly refer to a position $i$ as opposed to a letter $\pi_i$, and most authors do not take $\pi_{0}=\pi_{n+1}=\infty$ when defining these terms. It will be more convenient for us to adopt these conventions.

Next, recall that the $\rix$ statistic was defined by Lin using the recursive formula \eqref{e-rix}, and that Lin and Zeng later defined the set-valued statistic $\Rix$ for which $\rix$ gives the cardinality. The definition of $\Rix$ relies on Lin and Zeng's ``rix-factorization'', which is given below.

\begin{defn}  \label{defn:rix_lin_zeng}
 Each permutation $\pi \in \mathfrak{S}_n$ can be uniquely written in the form 
 \begin{equation}
     \pi = \alpha_1 \cdots \alpha_k \beta \label{eq:rix-factorization}
 \end{equation}
 where the factors $\alpha_1,\dots,\alpha_k,\beta$ (henceforth called \textit{rix-factors}) are obtained by applying the following algorithm:
 \begin{enumerate}
    \item[(1)] Initialize $w\coloneqq\pi$ and $i\coloneqq0$.
    \item[(2)] If $w$ is an increasing word, let $\beta \coloneqq w$ and terminate the algorithm. Otherwise, increase $i$ by 1, let $x$ be the largest descent of $w$, and write $w = w'xw''$ (so that $w'$ consists of all the letters of $w$ to the left of $x$, and $w''$ all the letters to the right of $x$).
    \item[(3)] If $w'= \varnothing$, let $\beta \coloneqq w$ and terminate the algorithm. Otherwise, let $\alpha_i \coloneqq w'x$ and $w \coloneqq w''$, and go to (2).
 \end{enumerate}
The expression \eqref{eq:rix-factorization} is called the \textit{rix-factorization} of $\pi$.
 \end{defn}

When writing out the rix-factorization of a permutation, we will often use vertical bars to demarcate the rix-factors. It will also be convenient for us to let $\beta_1(\pi)$ denote the first letter of $\beta$ in the rix-factorization of a permutation $\pi$.

\begin{defn} \label{d-rixpt}
A \emph{rixed point} of $\pi$ is a letter in the maximal increasing suffix of $\pi$ which is not smaller than $\beta_1(\pi)$, and the set of rixed points of $\pi$ is denoted $\Rix(\pi)$.
\end{defn}

\begin{example}\label{eg-rix1}
    Let us walk through the algorithm in Definition \ref{defn:rix_lin_zeng} for the permutation $\pi = 142785369$:
    \begin{itemize}
    	\item[(1)] Set $w=142785369$ and $i=0$. 
        \item[(2-1)] Since $w$ is not increasing, we set $i=1$, $x=8$, $w'=1427$, and $w''=5369$.
        \item[(3-1)] Since $w'\neq \varnothing$, we set $\alpha_1=14278$ and $w=5369$.
        \item[(2-2)] Since $w$ is not increasing, we set $i=2$, $x=5$, $w'=\varnothing$, and $w''=369$.
        \item[(3-2)] Since $w = \varnothing$, we set $\beta=5369$.
    \end{itemize}
    Thus the rix-factorization of $\pi$ is $14278|5369$ and $\Rix(\pi)=\{ 6,9 \}$.
\end{example}

The above example showcased a permutation for which the algorithm terminates in step (3). Below is an example in which termination occurs in step (2).

\begin{example}\label{eg-rix2}
    Let us walk through the algorithm in Definition \ref{defn:rix_lin_zeng} for the permutation $\pi = 23816457$:
    \begin{itemize}
    	\item[(1)] Set $w=23816457$ and $i=0$. 
        \item[(2-1)] Since $w$ is not increasing, we set $i=1$, $x=8$, $w'=23$, and $w''=16457$.
        \item[(3-1)] Since $w'\neq \varnothing$, we set $\alpha_1=238$ and $w=16457$.
        \item[(2-2)] Since $w$ is not increasing, we set $i=2$, $x=6$, $w'=1$, and $w''=457$.
        \item[(3-2)] Since $w'\neq \varnothing$, we set $\alpha_2=16$ and $w=457$.
        \item[(2-3)] Since $w$ is increasing, we set $\beta=457$.
    \end{itemize}
    Thus the rix-factorization of $\pi$ is $238|16|457$ and $\Rix(\pi)=\{ 4,5,7 \}$.
\end{example}

While the algorithm in Definition \ref{defn:rix_lin_zeng} is recursive, our algorithm in the next section is iterative.

\section{An iterative algorithm for rixed points and the rix-factorization}
	In this section, we give an iterative algorithm for computing the rixed points of a permutation along with its rix-factorization. This is achieved through the use of pointers on the permutation, restricting it to a \textit{valid factor}. At first the valid factor is taken to be the entire permutation, but we gradually restrict the valid factor as the algorithm progresses. To make this algorithm iterative, we consider all the entries of the permutation in decreasing order. For each entry $x$, we first check if it appears in the valid factor, and if it does, we use the local shape of the permutation around $x$ to move a boundary of the valid factor inward. When we move the left boundary, then a new term is added to the rix-factorization; when we move the right boundary, $x$ is added as a rixed point.

    After describing the algorithm explicitly, we will prove that the output of our algorithm indeed gives the rix-factorization and the rixed points as defined by Lin and Zeng, and then we use our algorithm to prove several more results concerning rixed points and the rix-factorization.
    
    \subsection{Explicit description of the algorithm}
    
    \begin{algorithm_env}\label{algorithm}
    Let $\pi\in \mathfrak{S}_n$ and $\sigma = \pi^{-1}$.
    Throughout this algorithm, let $\pi_l \cdots \pi_r$ denote the valid factor. We begin with $l=1$ and $r=n$, so that the entire permutation $\pi$ is the valid factor. 
    We let $x$ iterate through each of the letters $n, n-1, n-2, \ldots, 1$---in that order---until the stopping condition described below occurs. Let $i$ be the position of $x$ in $\pi$, so that $\pi_i=x$ or, equivalently, $\sigma_x = i$. If $x$ belongs to the valid factor (i.e., if $l \leq i \leq r$):
        \begin{enumerate} 
            \item[(a)] If $x$ is a peak of $\pi$, then the valid factor becomes $\pi_{i+1}\cdots \pi_r$, and $\pi_l\cdots \pi_i = \pi_l\cdots x$ is added to the rix-factorization.
            
            Otherwise, $x$ is either the first or the last letter of the valid factor (because letters are examined in decreasing order).
            \item[(b)]  If $x$ is the first letter of the valid factor, then we add $\pi_i\cdots \pi_n=x\cdots \pi_n$ to the rix-factorization, and $x$ is added as a rixed point if $x$ is an ascent of $\pi$. The algorithm terminates.
            \item[(c)]  If $x$ is the last letter (but not the first) of the valid factor, then the valid factor becomes $\pi_l \cdots \pi_{i-1}$, and $x$ is added as a rixed point.
        \end{enumerate}
        Once the algorithm stops (when $x$ is the first letter of the valid factor), we return the set of rixed points and the rix-factorization.
    \end{algorithm_env}

Pseudocode for Algorithm \ref{algorithm} is given below.
    
\begin{algorithm}[h]
\begin{algorithmic}[1]
\Ensure $\pi$ is a permutation
\State $ n \gets |\pi|$
\State $\sigma \gets \pi^{-1}$
\State $l \gets 1$
\State $r \gets n$
\State $\Rix \gets \{\}$
\State $\text{factorization} \gets []$
\State $\pi_0 \gets \infty,\ \pi_{n+1} \gets\infty$
\For{$x = n$ to $1$ (in decreasing order)}
\State $i\gets \sigma_x$
\If{$l\leq i \leq r$} \Comment{$x$ is in the valid factor}
\If{$\pi_{i-1} < x >  \pi_{i+1}$} \Comment{$x$ is a peak of $\pi$}
\State $l \gets i+1$
\State append $\pi_l\cdots x$ to factorization
\ElsIf{$i=l$}  \Comment{$x$ is the first letter of valid factor}
\State append $x\cdots \pi_n$ to factorization
\If{$x < \pi_{l+1}$} \Comment{$x$ is an ascent}
\State add $x$ to $\Rix$
\EndIf
\State \Return{$\Rix$, factorization}
\Else \Comment{$x$ is the last letter (but not the first) of the valid factor}
    \State add $x$ to $\Rix$
    \State $r \gets i-1$
\EndIf
\EndIf
\EndFor
\end{algorithmic}
\renewcommand{\thealgorithm}{5}
\caption{Compute rixed points and rix-factorization}\label{pseudocode}
\end{algorithm}
\begin{remark}
Algorithm \ref{algorithm} is executed in a time that is linear with respect to $n$. The number of operations is at least linear, since finding the inverse of a permutation requires reading it all and is thus executed in linear time. All other operations (comparisons, attributions, additions to list) are done in constant time, and there is a single for-loop (that we go through at most $n$ times), meaning that the algorithm requires a number of operations that is at most linear with respect to $n$.
\end{remark}

Let us illustrate this algorithm with two examples; compare with Examples \ref{eg-rix1}--\ref{eg-rix2}.

    \begin{example}[Example \ref{eg-rix1} continued] \label{eg-rix1-2}
        Let us walk through Algorithm \ref{algorithm} for $\pi = 142785369$, highlighting the evolution of the valid factor:
    	\begin{itemize}[leftmargin=45bp]
    		\item[($x=9$)] $14278536\mathbf{9} \to 14278536$, because $x=9$ is at the 
    		end of the valid factor. The rix-factorization contains no terms yet, and $9$ is added as a rixed point.
    		\item[($x=8$)] $1427\mathbf{8}536 \to 536$, because $x=8$ is a peak. We add $14278$ to the rix-factorization, and the set of rixed points is unchanged.
    		\item[($x=7$)] $536 \to 536$, because $x=7$ is outside the valid factor. The rix-factorization and the set of rixed points are unchanged.
    		\item[($x=6$)] $53\mathbf{6} \to 53$, because $x=6$ is at the end of the 
    		valid factor. The rix-factorization is unchanged, and $6$ is added as a rixed point.
    		\item[($x=5$)] The algorithm terminates because $x=5$ is the first letter of the valid factor $\mathbf{5}3$. We add $5369$ to the rix-factorization, but $5$ is not added as a rixed point because it is not an ascent.
    	\end{itemize}
    	Thus the rix-factorization of $\pi$ is $14278|5369$ and $\Rix(\pi)=\{ 6,9 \}$, which agrees with what was obtained before.
    \end{example}
    
    \begin{example}[Example \ref{eg-rix2} continued] \label{eg-rix2-2}
        Let us walk through Algorithm \ref{algorithm} for $\pi = 23816457$, highlighting the evolution of the valid factor:
    	\begin{itemize}[leftmargin=45bp]
    		\item[($x=8$)] $23\mathbf{8}16457 \to 16457$, because $x=8$ is a peak. We add $238$ to the rix-factorization, and the set of rixed points is currently empty.
    		\item[($x=7$)] $1645\mathbf{7} \to 1645$, because $x=7$ is at the end of the valid factor. The rix-factorization is unchanged, and $7$ is added as a rixed point.
    		\item[($x=6$)] $1\mathbf{6}45\to 45$, because $x=6$ is a peak. We add $16$ to the rix-factorization, and the set of rixed points is unchanged.
    		\item[($x=5$)] $4\mathbf{5} \to 4$, because $x=5$ is at the end of the valid factor. The rix-factorization is unchanged, and $5$ is added as a rixed point.
        	\item[($x=4$)] The algorithm terminates because $x=4$ is the first letter of the valid factor $\mathbf{4}$. We add $457$ to the rix-factorization, and $4$ is added as a rixed point because it is an ascent.
    	\end{itemize}
        Thus the rix-factorization of $\pi$ is $238|16|457$ and $\Rix(\pi)=\{ 4,5,7 \}$, which agrees with what was obtained before.
    \end{example}

See Figure \ref{f-alg} for visual depictions of Examples \ref{eg-rix1-2} and \ref{eg-rix2-2}.
    
\begin{figure}[t]
\begin{center}
\begin{tikzpicture}[scale=0.5] 	
    \draw[step=1,lightgray,thin] (0,1) grid (10,10); 
	\tikzstyle{node0}=[circle, inner sep=2, fill=black] 
	\tikzstyle{node1}=[rectangle, inner sep=3, fill=dartmouthgreen] 
	\tikzstyle{node2}=[diamond, inner sep=2, fill=davidsonred] 
	\node[node0] (A) at (1,1) {}; 
	\node[node0] (B) at (2,4) {}; 
	\node[node0] (C) at (3,2) {}; 
	\node[node0] (D) at (4,7) {}; 
	\node[node1] (E) at (5,8) {}; 
	\node[node0] (F) at (6,5) {}; 
	\node[node0] (G) at (7,3) {}; 
	\node[node2] (H) at (8,6) {};  
    \node[node2] (I) at (9,9) {}; 
	\tikzstyle{ridge}=[draw, line width=1, dotted, color=black] 
	\path[ridge] (0,9)--(A) -- (B) -- (C) -- (D) -- (E) -- (F) -- (G) -- (H) -- (I) --(10,10); 
    \tikzstyle{hop1}=[draw, line width = 1, color=davidsonred,-]
	\path[hop1] (0.5,12.5)--(9.5,12.5);
    \node[node2] at (9,12) {}; 
    \path[hop1] (0.5,12)--(8.5,12);
    \node[node1] at (5,11.5) {};
	\path[hop1] (5.5,11.5)--(8.5,11.5);
    \path[hop1] (5.5,11)--(8.5,11);
    \node[node2] at (8,10.5) {}; 
	\path[hop1] (5.5,10.5)--(7.5,10.5);
    \node[color=davidsonred] at (0,13) {\small trace of valid factor};

    \tikzstyle{pi}=[above=-5] 
	\node[pi] at (1,0) {1}; 
	\node[pi] at (2,0) {4}; 
	\node[pi] at (3,0) {2}; 
	\node[pi] at (4,0) {7}; 
	\node[pi] at (5,0) {8}; 
	\node[pi] at (6,0) {5}; 
	\node[pi] at (7,0) {3}; 
	\node[pi] at (8,0) {6}; 
    \node[pi] at (9,0) {9};

	\begin{scope}[shift={(16,0)}]
	\draw[step=1,lightgray,thin] (0,1) grid (9,9); 
    \tikzstyle{node0}=[circle, inner sep=2, fill=black] 
	\tikzstyle{node1}=[rectangle, inner sep=3, fill=dartmouthgreen] 
	\tikzstyle{node2}=[diamond, inner sep=2, fill=davidsonred] 
	\node[node0] (A) at (1,2) {}; 
	\node[node0] (B) at (2,3) {}; 
	\node[node1] (C) at (3,8) {}; 
	\node[node0] (D) at (4,1) {}; 
	\node[node1] (E) at (5,6) {}; 
	\node[node2] (F) at (6,4) {}; 
	\node[node2] (G) at (7,5) {}; 
	\node[node2] (H) at (8,7) {};  
    \tikzstyle{ridge}=[draw, line width=1, dotted, color=black] 
	\path[ridge] (0,9)--(A) -- (B) -- (C) -- (D) -- (E) -- (F) -- (G) -- (H) -- (9, 9); 
    \tikzstyle{pi}=[above=-5] 
    \node[pi] at (1,0) {2}; 
	\node[pi] at (2,0) {3}; 
	\node[pi] at (3,0) {8}; 
	\node[pi] at (4,0) {1}; 
	\node[pi] at (5,0) {6}; 
	\node[pi] at (6,0) {4};
	\node[pi] at (7,0) {5}; 
	\node[pi] at (8,0) {7}; 
    \tikzstyle{hop1}=[draw, line width = 1, color=davidsonred,-]
	\path[hop1] (0.5,12.5)--(8.5,12.5);
    \path[hop1] (3.5,12)--(8.5,12);
	\path[hop1] (3.5,11.5)--(7.5,11.5);
    \path[hop1] (5.5,11)--(7.5,11);
	\path[hop1] (5.5,10.5)--(6.5,10.5);
    \node[color=davidsonred] at (0,13) {\small trace of valid factor};
    \node[node1]  at (3,12) {}; 
	\node[node1] at (5,11) {}; 
	\node[node2] at (6,10) {}; 
	\node[node2] at (7,10.5) {}; 
	\node[node2] at (8,11.5) {};  
    
	\end{scope}
\end{tikzpicture}
\end{center}

\caption{Visual depictions of Algorithm \ref{algorithm} on $\pi = 142785369$ and $\pi = 23816457$, as in Examples~\ref{eg-rix1-2} and \ref{eg-rix2-2}. The peaks that mark the end of each $\alpha$-factor in the rix-factorization are depicted as green squares, and the rixed points as red diamonds. The progression of the valid factor is depicted on top of the permutation.}\label{f-alg}
\end{figure}
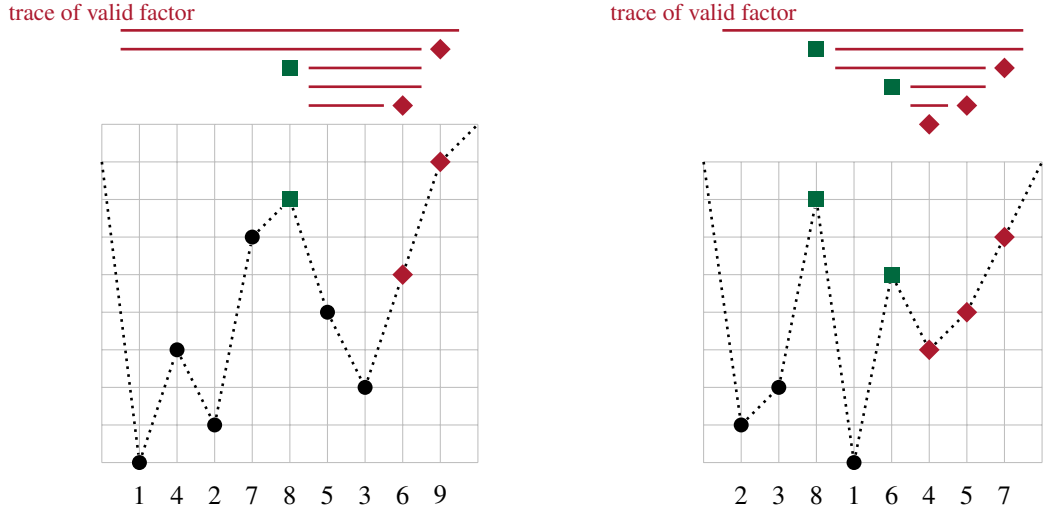

\subsection{Proof that Algorithm \ref{algorithm} gives the rix-factorization}
We first show that Algorithm \ref{algorithm} indeed gives what Lin and Zeng defined to be the rix-factorization (as in Definition \ref{defn:rix_lin_zeng}). To do so, we need a lemma regarding what is to the right of the valid factor.

\begin{lem}\label{lem:ascents_right_of_valid_factor}
    At any stage during the execution of Algorithm \ref{algorithm}, if the valid factor is $\pi_l\cdots\pi_r$, then $\pi_r, \pi_{r+1}, \ldots, \pi_n$ are all ascents of $\pi$, so $\pi_r< \pi_{r+1} < \cdots < \pi_n$.
\end{lem}
\begin{proof}
    We proceed by induction. Our base case is at the start of the algorithm, at which point there is nothing to the right of the valid factor. Since the last letter of a permutation is defined to be an ascent by convention, the result holds.
    
    Now, assume that the statement is true when the valid factor is $\pi_l\cdots\pi_r$. We need to show that iterating Algorithm \ref{algorithm} by one step preserves the accuracy of the statement. For this, we show that whenever the right boundary of the valid factor moves, it moves only by one position and the new right boundary is an ascent of $\pi$. By Algorithm \ref{algorithm}, the right boundary changes only when $\pi_r$ is the largest letter of the valid factor $\pi_l\cdots\pi_r$ and $l\neq r$, at which point the valid factor becomes $\pi_l\cdots\pi_{r-1}$. Hence, $\pi_{r-1} < \pi_r$, so $\pi_{r-1}$ is an ascent of $\pi$. We also know from the induction hypothesis that $\pi_{r}, \pi_{r+1}, \ldots, \pi_{n}$ are ascents. Thus, everything weakly to the right of the right boundary $\pi_{r-1}$ is an ascent, or equivalently, $\pi_{r-1}<\pi_{r}<\cdots < \pi_{n}$.
\end{proof}

\begin{prop}\label{prop:same_rix_factorization}
    Algorithm \ref{algorithm} produces the rix-factorization as given in Definition \ref{defn:rix_lin_zeng}. 
\end{prop}

\begin{proof}
    To prove the proposition, we simultaneously apply to a permutation $\pi$ the procedures in Definition \ref{defn:rix_lin_zeng} and in Algorithm~\ref{algorithm}, showing that the terms added to the rix-factorization are the same at each step.

    Algorithm \ref{algorithm} considers a factor of $\pi$ (seen as a word), called the valid factor. Let $w$ be as in Definition \ref{defn:rix_lin_zeng} and let $v$ be the valid factor of $\pi$ plus what is to its right. At first, we have $w=v=\pi$. We show, by induction, that $v=w$ at each step of the joint execution of the procedure in Definition~\ref{defn:rix_lin_zeng} and that of Algorithm \ref{algorithm}. The base case is $v=w=\pi$, and no term is in the rix-factorization at this stage. For the induction hypothesis, we assume $v=w$, and we apply the recursive procedure in Definition \ref{defn:rix_lin_zeng} and the iterative procedure in Algorithm \ref{algorithm}. Consider the following two cases:
    \begin{enumerate}
        \item[(1)] Suppose that $w$ has a descent; then it is not an increasing word. Let $y$ be the largest descent of $w$. We have $v=w$ by the induction hypothesis, so $y$ is also the largest descent of $v$. Also, let $x$ be the largest letter of the valid factor. There are two options for $x$: either it is an ascent or it is a descent. If $x$ is a descent, then it must be the largest descent in $v$, and since what is to the right of the valid factor consists only of ascents (by Lemma \ref{lem:ascents_right_of_valid_factor}), we have $y = x$. If $x$ is an ascent and is the largest letter of the valid factor, then it must be the rightmost letter of the valid factor (otherwise, the letter to its right is a larger letter in the valid factor). Hence, in Algorithm \ref{algorithm}, the right boundary of the valid factor moves one position to the left, and $v$ is untouched. We repeat the process until the largest letter is a descent, making $x=y$.

        Let us write $w=w'yw''$. Following the procedure in Definition \ref{defn:rix_lin_zeng}, if $w'$ is empty, the algorithm stops and $\beta = w$ is added to the rix-factorization as the last rix-factor. Otherwise, if $w'$ is not empty, then we add $w'y$ to the rix-factorization and repeat the process with $w''$ in place of $w$. Also, let us write $v$ as $v'xv''$. Following Algorithm \ref{algorithm}, if $v'$ is empty then $x$ is the first letter of the valid factor, so $\beta=v$ is added to the rix-factorization and the algorithm stops. If $v'$ is not empty, then $x$ is a peak, so we add $v'x$ to the rix-factorization and repeat the process with $v''$ in place of $v$. Since $x=y$ and $v=w$, we have $v'=w'$ and $v''=w''$; thus the same terms have been added to the rix-factorization, concluding the induction step in this case.
        \item[(2)] If $w$ has no descent, then neither does $v$, so they are both increasing words. In that case, Definition \ref{defn:rix_lin_zeng} sets $\beta=w$ and terminates the process. As for $v$, since it is increasing, its largest letter is successively the largest letter of the valid factor. Thus, during the execution of Algorithm \ref{algorithm}, the right boundary of the valid factor moves one step to the left at a time, which does not impact $v$. Hence, the process in Algorithm \ref{algorithm} is repeated until the valid factor has a single letter, in which case the algorithm stops and we add $\beta = v$ to the rix-factorization.
    \end{enumerate}
    By these two cases, we have shown inductively that the terms of the rix-factorization obtained using Definition \ref{defn:rix_lin_zeng} and Algorithm \ref{algorithm} are the same, thus completing the proof.
\end{proof}

\subsection{Proof that Algorithm \ref{algorithm} gives the rixed points}
Recall from Definition \ref{d-rixpt} that the set of rixed points of a permutation $\pi$ is defined as a letter in the maximal increasing suffix of $\pi$ that is no smaller than $\beta_1(\pi)$, the first letter of the rix-factor $\beta$ of $\pi$. We show here that the set of rixed points obtained from Algorithm \ref{algorithm} is indeed the same set.

\begin{prop}
Algorithm \ref{algorithm} produces the set of rixed points as given in Definition \ref{d-rixpt}. 
\end{prop}
\begin{proof}
    We let $A$ be the set of letters in the maximal increasing suffix of $\pi$ that are no smaller than the first letter of $\beta_1(\pi)$ (so $A$ is the set of rixed points as obtained from Definition \ref{d-rixpt}). Also, let $B$ contain the successive right boundaries of the valid factor of $\pi$, as well as the left boundary of the valid factor when Algorithm \ref{algorithm} terminates if it is an ascent (so $B$ is the set of rixed points as obtained from Algorithm \ref{algorithm}). We show that $A = B$.

    Let $y \in A$. Then $y$ is larger than or equal to $\beta_1(\pi)$, which is the stopping point of Algorithm \ref{algorithm}. Hence, $y$ is considered during the execution of the Algorithm \ref{algorithm}. We also know that $y$ is part of the valid factor when it is considered, since the right boundary only excludes letters after consideration, and the left boundary only moves to the right of peaks. However, since $y$ is part of an increasing suffix, it cannot have a peak to its right nor can it be a peak itself (or a descent in general). In particular, the fact that $y$ is not a peak also implies that $y$ must either be the first or the last letter of the valid factor when it is considered. If $y$ is the first letter of the valid factor, then the algorithm stops, and $y$ is in $B$ because it is an ascent (as it is part of an increasing suffix). Otherwise, if $y$ is the last letter (but not the first), then the algorithm still adds $y$ to $B$. In any case, $y \in B$.

    We now prove the other inclusion. Let $z \in B$. Then, $z$ is either the first letter of the valid factor when the algorithm stops, or it is the last letter of the valid factor at some point during the execution of the algorithm. In the latter case, $z$ being at the end of the valid factor means that everything to its right is greater than $z$ (by Lemma \ref{lem:ascents_right_of_valid_factor}), so it is part of an increasing suffix of $\pi$. If $z$ is the first letter of the valid factor when the algorithm stops, then it is in $B$ only if is an ascent of $\pi$, in which case $z$ is also the last letter of the valid factor by the argument in Case (1) of the proof of Proposition \ref{prop:same_rix_factorization}. In either case, $z$ belongs to the maximal increasing suffix of $\pi$. Moreover, 
    we know that $z$ is considered by Algorithm \ref{algorithm} during its execution and that $\beta_1 (\pi)$ is the last letter considered prior to termination. Since the algorithm considers the letters in $\pi$ in decreasing order of their values, it follows that $z\geq \beta_1 (\pi)$. Therefore, $z\in A$.

    We have thus proved that $A = B$, so Algorithm \ref{algorithm} indeed gives the set of rixed points.
    \end{proof}
    
    We now show that the recursive definition \eqref{e-rix} for the $\rix$ statistic can be adapted to obtain a recursive algorithm for computing the set of rixed points. As a consequence, we get that $\rix(\pi)$ is indeed the cardinality of $\Rix(\pi)$ for any permutation $\pi$.\footnote{This was stated by Lin and Zeng \cite[Proposition 9 (iii)]{Lin2015} but no proof was given.}
    
    \begin{prop} \label{p-Rixrecur}
    Given a word $w = w_1 w_2 \cdots w_k$ with distinct positive integer letters, define $\Rix(w)$ in the following way: If $k=0$ \textup{(}i.e., if $w=\varnothing$\textup{)}, then $\Rix(w) = \varnothing$. Otherwise, if $w_i$ is the largest letter of $w$, then let
	\begin{equation}
	\Rix(w)\coloneqq
	\begin{cases}
		\varnothing, & \text{if }i=1<k,\\
		\{w_k\}\cup\Rix(w_1w_2\cdots w_{k-1}), & \text{if }i=k,\\
		\Rix(w_{i+1}w_{i+2}\cdots w_k), & \text{if }1<i<k.
	\end{cases} \label{e-Rix}
    \end{equation}
    Then, for any permutation $\pi$, the set $\Rix(\pi)$ obtained using the above recursive definition is indeed the set of rixed points of $\pi$.
    \end{prop}
   
    Before proving Proposition \ref{p-Rixrecur}, let us illustrate this recursive algorithm for $\Rix$ with a couple examples. By comparing these with Examples \ref{eg-rix1-2}--\ref{eg-rix2-2}, which compute the rixed points of the same permutations but using Algorithm \ref{algorithm}, we see that the valid factors of Algorithm \ref{algorithm} are precisely the words $w$ at each step of \eqref{e-Rix}. This will be key to our proof of Proposition \ref{p-Rixrecur}.

    \begin{example}[Examples \ref{eg-rix1} and \ref{eg-rix1-2} continued] \label{eg-rix1-3}
       We use Proposition \ref{p-Rixrecur} to compute the rixed points of $\pi = 142785369$:
        \begin{align*}
        \Rix(14278536\mathbf{9}) &= \{9\} \cup \Rix(1427\mathbf{8}536) \\
                        &= \{9\} \cup \Rix(53\mathbf{6}) \\
                        &= \{6,9\} \cup \Rix(\mathbf{5}3) \\
                        &= \{6,9\}.
        \end{align*}
    \end{example}

    \begin{example}[Examples \ref{eg-rix2} and \ref{eg-rix2-2} continued] \label{eg-rix2-3}
        We use Proposition \ref{p-Rixrecur} to compute the rixed points of $\pi = 23816457$:
        \begin{align*}
        \Rix(23\mathbf{8}16457) &= \Rix(1645\mathbf{7}) \\
                        &= \{7\} \cup \Rix(1\mathbf{6}45) \\
                        &= \{7\} \cup \Rix(4\mathbf{5}) \\
                        &= \{5,7\} \cup \Rix(\mathbf{4}) \\
                        &= \{4,5,7\}.
        \end{align*}
    \end{example}
    
    \begin{proof}[of Proposition \ref{p-Rixrecur}]
         We show using induction that the three cases in \eqref{e-Rix} correspond (with a slight adjustment) to the three cases in Algorithm \ref{algorithm}, and that the word $w$ changes in the same way as the valid factor in 
         Algorithm \ref{algorithm}. Both procedures begin with the entire permutation $\pi$, which establishes the base case. For our induction hypothesis, suppose that $w = w_1\cdots w_{k}$ is the valid factor of $\pi$. Let $w_i$ be the largest letter of $w$. Consider the following cases:
        \begin{itemize}
            \item If $w_i$ is the first letter (but not the last) of $w$, then there are no rixed points in $w$ according to \eqref{e-Rix}. Note that if $w_i$ is both an ascent of $\pi$ and the first letter of $w$, then it is the only (and therefore last) letter of $w$ because it is the largest letter in $w$. Hence, this case corresponds to case (b) of Algorithm \ref{algorithm} when $w_i$ is not an ascent of $\pi$, in which no rixed points are added and the algorithm terminates.
            \item If $w_i$ is the last letter of $w$, then \eqref{e-Rix} adds $w_i$ as a rixed point and $w$ becomes $w_1\cdots w_{k-1}$. Note that if $w_i$ is the only letter of $w$, then $w_1\cdots w_{k-1} = \varnothing$ and the procedure in \eqref{e-Rix} stops. This corresponds to case (c) of Algorithm \ref{algorithm}, as well as case (b) when $w_i$ is the only letter of $w$ (and thus an ascent of $\pi$ by Lemma \ref{lem:ascents_right_of_valid_factor}).
            \item Otherwise, the largest letter $w_i$ of $w$ is neither its first nor last (which means that $w_i$ is a peak of $\pi$), so $w$ becomes $w_{i+1}\cdots w_k$ and no rixed point is added. This corresponds to case (a) of Algorithm \ref{algorithm}.
        \end{itemize}
        Because the valid factor and the set of rixed points change in the same way at each step of both algorithms, we get the same set of rixed points at the end.
    \end{proof}

\subsection{A characterization of rixed points}

If $y$ is a descent of $\pi$, let us call $y$ a \textit{leading descent} of $\pi$ if there is no larger descent of $\pi$ appearing after $y$. For example, the leading descents of $\pi=194376528$ are $9$, $7$, $6$, and $5$, whereas $4$ is a descent of $\pi$ but not a leading descent because the larger letters $7$, $6$, and $5$ are all descents of $\pi$ that appear after $4$.

Leading descents are useful in the study of rixed points. For example, note that the $x$ in step (2) of Definition \ref{defn:rix_lin_zeng} is always a leading descent of $\pi$. We will now use leading descents to define the ``maximal descending ridge'' of a permutation, which plays a role in our subsequent characterization of rixed points.

 \begin{defn}
     In this definition, we will prepend $\infty$ to $\pi$ and consider $\infty$ to be the first leading descent of $\infty\pi$. We define the \textit{maximal descending ridge} of a permutation $\pi$ in the following way:
     \begin{itemize}
        \item If $\pi_i$ and $\pi_{i+1}$ are the leftmost pair of leading descents in adjacent positions, then the maximal descending ridge of $\pi$ is the prefix of $\infty\pi$ ending with $\pi_i$.
        \item If $\infty\pi$ does not have two leading descents in adjacent positions, then the maximal descending ridge of $\pi$ is the prefix of $\infty\pi$ ending with the rightmost descent of $\pi$.
     \end{itemize}
 \end{defn}

\noindent To illustrate, the maximal descending ridges of the permutations $142785369$ and $23816457$ (from the earlier examples) are $\infty 14278$ and $\infty 23816$, respectively.

\begin{thm}
\label{t-rixptchar}
    Let $\pi \in \mathfrak{S}_n$. The rixed points of $\pi$ are characterized by the following:
    \begin{enumerate}
        \item[\normalfont{(a)}] If $y\in \Rix(\pi)$, then either $y$ is a double ascent of $\pi$, or $y$ is the rightmost valley of $\pi$ and is either the first letter of $\pi$ or immediately follows a peak in $\pi$.
        \item[\normalfont{(b)}] A letter $y$ satisfying the above conditions is a rixed point of $\pi$ if and only if none of $y+1, y+2, \ldots, n$ is immediately to the right of the maximal descending ridge of $\pi$ or appears as a peak to the right of $y$.
    \end{enumerate}
\end{thm}

Before giving the proof, let us use the characterization given by Theorem \ref{t-rixptchar} to determine the rixed points of the two permutations from the earlier examples.

    \begin{example}[Examples \ref{eg-rix1}, \ref{eg-rix1-2}, and \ref{eg-rix1-3} continued] \label{eg-rix1-4}
        Take $\pi = 142785369$, whose maximal descending ridge is $\infty 14278$. The letters of $\pi$ meeting the requirements in Theorem \ref{t-rixptchar} (a) are $7$, $6$, and $9$, but $7$ is not a rixed point because $8$ is a peak immediately following $7$.\footnote{We also know that $7$ cannot be a rixed point because it is not part of an increasing suffix of $\pi$.} On the other hand, $6$ and $9$ are both rixed points because none of $7$, $8$, and $9$ is immediately to the right of the maximal descending ridge or is a peak to the right of $6$, and there is no letter in $\pi$ larger than $9$. Thus, $\Rix(\pi)=\{ 6,9 \}$.
    \end{example}

    \begin{example}[Examples \ref{eg-rix2}, \ref{eg-rix2-2}, and \ref{eg-rix2-3} continued] \label{eg-rix2-4}
        Take $\pi = 23816457$, whose maximal descending ridge is $\infty 23816$. The letters of $\pi$ meeting the requirements in Theorem \ref{t-rixptchar} (a) are $3$, $4$, $5$, and $7$, and then it is readily verified that $3$ is the only one which does not meet the requirements in Theorem \ref{t-rixptchar} (b). Thus, $\Rix(\pi)=\{ 4,5,7 \}$.
    \end{example}

\begin{remark}
In Examples \ref{eg-rix1-4} and \ref{eg-rix2-4}, the letter following the maximal descending ridge of $\pi$ coincides with the beginning of the $\beta$ rix-factor of $\pi$. This will be confirmed in the proof of Theorem~\ref{t-rixptchar} below, as we shall show that if Algorithm \ref{algorithm} terminates while considering $x$ (so that $x$ is the first letter $\beta_1 (\pi)$ of $\beta$), then $x$ immediately follows the maximal descending ridge of $\pi$. Therefore, we can also compute the rixed points of a permutation $\pi$ by using the maximal descending ridge to determine $\beta_1 (\pi)$ and then applying Definition \ref{d-rixpt}.
\end{remark}

 \begin{proof}[of Theorem \ref{t-rixptchar}]
     We first prove (a). By Definition \ref{d-rixpt}, every rixed point $y$ of $\pi$ belongs to an increasing suffix of $\pi$, so $y$ is an ascent of $\pi$. Every ascent is either a double ascent or a valley, and if $y$ is a valley, then it is the rightmost valley of $\pi$ as it must be the first letter of the maximal increasing suffix of $\pi$. Furthermore, in Algorithm \ref{algorithm}, we see that a valley $y$ is added as a rixed point only when it is the first letter of the valid factor; if $y$ is instead the last letter (but not the first) of the valid factor, then $y$ would not be the largest letter of the valid factor, contradicting the fact that Algorithm \ref{algorithm} considers the letters of $\pi$ in decreasing order. Since $y$ is the first letter of the valid factor, it follows that it is either the first letter of $\pi$ or it is immediately to the right of a peak (because this is how the left boundary is moved in Algorithm \ref{algorithm}). Hence, part (a) is proven. 
     
     To prove (b), let us first assume that $y$ is a rixed point of $\pi$, and show that none of $y+1, y+2, \ldots, n$ is immediately to the right of the maximal descending ridge of $\pi$ or appears as a peak to the right of $y$. For this, recall again that the letters of $\pi$ are inspected in decreasing order by Algorithm \ref{algorithm} until we reach the stopping condition, which is when the largest letter of the valid factor is its first letter. Let $q$ be the largest (and thus first) letter of the valid factor when the algorithm stops. None of $q+1, q+2, \ldots, n$ is a peak to the right of $q$; otherwise, the left boundary of the valid factor would have been moved to the right of $q$, and so $q$ would not be in the valid factor. We also know that $y$ is weakly to the right of $q$ and that $y\geq q$ because $y$ is a rixed point, so $\{y+1, y+2, \ldots, n\}$ is a subset of $\{q+1, q+2, \ldots, n \}$, and thus it follows from the analogous statement for $q$ that none of $y+1, y+2, \ldots, n$ is a peak to the right of $y$.
     
     By the same reasoning as in part (a), either $q$ is the first letter of $\pi$, or it immediately follows a peak of $\pi$---call it $z$. In the latter case, we claim that $z$ is the last letter of the maximal descending ridge of $\pi$, which we prove in the following steps:
     \begin{enumerate}
         \item[(1)] We show that $z$ is a leading descent of $\pi$. Otherwise, if $z^{\prime}$ were the largest descent to the right of $z$ that is greater than $z$, then $z^{\prime}$ would be a peak, and so the left boundary of Algorithm~\ref{algorithm} would have moved to the right of $z^{\prime}$ upon considering $z^{\prime}$. Hence, no such $z^{\prime}$ exists, so $z$ is indeed a leading descent.
         \item[(2)] We show that there cannot be leading descents in adjacent positions weakly to the left of $z$, which would imply that $z$ belongs to the maximal descending ridge. Suppose otherwise, and let $z^{\prime}$ and $z^{\prime\prime}$ be the leftmost pair of leading descents in adjacent positions. (Note that $z^\prime$ can be $\infty$, and $z^{\prime\prime}$ can be $z$ unless $z^\prime=\infty$.) Since $z^{\prime} > z^{\prime\prime} \geq z > q$ and Algorithm \ref{algorithm} terminates while considering $q$, the algorithm must consider $z^{\prime}$ (unless $z^\prime=\infty$) and $z^{\prime\prime}$ at some point. If $z^\prime=\infty$, then $z^{\prime\prime}$ is the first letter of $\pi$ and thus the first letter of the valid factor at the beginning of Algorithm $\ref{algorithm}$. If $z^{\prime}\neq \infty$, then $z^{\prime}$ is a peak, so the left boundary of the valid factor would move to $z^{\prime\prime}$ upon considering $z^{\prime}$. Either way, the valid factor will begin with $z^{\prime\prime}$ at some point during the execution of the algorithm. And since $z^{\prime\prime}$ is a leading descent, there are no peaks larger than $z^{\prime\prime}$ to the right of $z^{\prime\prime}$, so the left boundary would stay at $z^{\prime\prime}$ until $z^{\prime\prime}$ is considered by the algorithm, at which point the algorithm terminates. This contradicts the assumption that the algorithm terminates while considering $q$, so no such $z^{\prime}$ and $z^{\prime\prime}$ exist. 
         \item[(3)] If $q$ is a descent, then $q$ is a leading descent by the same reasoning as in (1), so $z$ is the last letter of the maximal descending ridge. If $q$ is an ascent, then it is added by Algorithm \ref{algorithm} as a rixed point, and since the rixed points form an increasing suffix of $\pi$, this means that $z$ is the rightmost descent of $\pi$ and therefore the last letter of the maximal descending ridge.
     \end{enumerate}
    Note that if $q$ is the first letter of $\pi$, then the argument in (3) suffices to show that the maximal descending ridge is $\infty$. In either case, $q$ is immediately to the right of the maximal descending ridge, and since $y\geq q$, it follows that none of $y+1, y+2, \ldots, n$ is immediately to the right of the maximal descending ridge. 

     Conversely, suppose that $y$ satisfies the conditions in part (a), and that none of $y+1, y+2, \ldots, n$ immediately follows the maximal descending ridge of $\pi$ or appears as a peak to the right of $y$. We wish to show that $y$ is a rixed point of $\pi$. First, note that the left boundary of the valid factor never moves to the right of $y$ during the execution of Algorithm \ref{algorithm}, since no letter larger than $y$ is a peak to the right of $y$. Also, if the algorithm were to terminate while considering a letter $q$ larger than $y$, then we know from an argument earlier in this proof that $q$ immediately follows the maximal descending ridge, which is a contradiction. Hence, $y$ is considered by Algorithm \ref{algorithm} at some point during its execution, and $y$ appears in the valid factor while under consideration.

     Now, recall that the conditions in part (a) imply that $y$ is an ascent and therefore not a peak, so $y$ is either the first or the last letter of the valid factor when it is considered by the algorithm. If $y$ is the first letter, then the fact that it is an ascent guarantees that it is a rixed point. If $y$ is the last letter (but not the first) of the valid factor, then Algorithm \ref{algorithm} adds it as a rixed point as well. Thus the proof of part (b) is complete.
\end{proof} 
\subsection{Properties of rixed points and the rix-factorization}

Before proceeding, we shall give a few more properties of rixed points and the rix-factorization which will be used later in this paper.

\begin{lem} 
\label{l-valdasc}
Let $\pi$ be a permutation with rix-factorization $\pi=\alpha_1\cdots\alpha_k\beta$. Let $y\in \Rix(\pi)$. Then $y$ is a valley of $\pi$ if $y=\beta_1(\pi)$, and is a double ascent of $\pi$ otherwise.
\end{lem}

\begin{proof}
Recall from Theorem \ref{t-rixptchar} (a) that a rixed point of $\pi$ is either a valley or a double ascent of $\pi$. By the algorithm in Definition \ref{defn:rix_lin_zeng}, $\beta_1(\pi)$ must either be the first letter of $\pi$ or is immediately preceded by a descent. Hence, if $y=\beta_1(\pi)$, then $y$ is a valley of $\pi$.

Now, suppose that $y\neq\beta_1(\pi)$, so that $y>\beta_1(\pi)$. Then $y$ is added to the set of rixed points in Algorithm \ref{algorithm} when it is the last letter of the valid factor, which is subsequently set to $\pi_l\cdots\pi_r$ where $\pi_r$ is the letter immediately preceding $y$. If $\pi_r$ is a peak of $\pi$, then in particular $\pi_r>y$ and so the algorithm must have examined $x=\pi_r$ prior to $x=y$, at which stage the valid factor would have been set to begin with $y$. This means that the algorithm would terminate at $x=y$ and $y$ would be the first letter of $\beta$, a contradiction. Therefore, $y$ is not the first letter of $\pi$ and does not immediately follow a peak, so $y$ is a double ascent of $\pi$ by Theorem \ref{t-rixptchar}.
\end{proof}

\begin{lem} \label{l-leaddes}
Let $\pi$ be a permutation with rix-factorization $\pi=\alpha_1\cdots\alpha_k\beta$. If $y$ is a leading descent of $\pi$, then either $y$ is an entry of $\beta$ or $y$ is the last letter of $\alpha_i$ for some $1\leq i\leq k$.
\end{lem}

\begin{proof}
Suppose that $y$ is a leading descent of $\pi$ but does not belong to $\beta$. Then $y$ belongs to a rix-factor $\alpha_i$ of $\pi$. It is evident from both Definition \ref{defn:rix_lin_zeng} and Algorithm \ref{algorithm} that the last letter of $\alpha_i$ is the largest letter of $\alpha_i$ and is a descent of $\pi$. So, if $y$ were not the last letter of $\alpha_i$, then the last letter of $\alpha_i$ would be a descent of $\pi$ to the right of $y$ which is larger than $y$, contradicting the assumption that $y$ is a leading descent of $\pi$.
\end{proof}

\begin{lem} \label{l-xbet} 
Let $\pi$ be a permutation with rix-factorization $\pi=\alpha_1 \alpha_2 \cdots \alpha_k \beta$. The following are
equivalent:
\begin{enumerate}
\item[\normalfont(a)]  The rix-factor $\beta$ is increasing.
\item[\normalfont(b)]  Every letter of $\beta$ is a rixed point of $\pi$.
\item[\normalfont(c)]  Every letter of $\beta$ is either a rixed point of $\pi$ or is $\beta_1(\pi)$.
\item[\normalfont(d)]  $\beta_1(\pi)$ is a rixed point of $\pi$.
\end{enumerate}
Furthermore, let $y$ be any letter of $\beta$ that is neither $\beta_1(\pi)$ nor a rixed point. Then $y<\beta_1(\pi)$.
\end{lem}

\begin{proof}
The equivalences (a)--(b) and (a)--(d) are immediate from Definition \ref{d-rixpt}, and (b) obviously implies (c). If every letter of $\beta$ appearing after $\beta_1(\pi)$ is a rixed point of $\pi$, then $\beta_1(\pi)$ is also part of the maximal increasing suffix consisting of letters not smaller than $\beta_1(\pi)$, so $\beta_1(\pi)$ is also a rixed point of $\pi$. Thus (c) implies (b).

Now, let $y$ be a letter of $\beta$ that is neither $\beta_1(\pi)$ nor a rixed point. Then it is easy to see that $y$ is a non-initial letter of the valid factor when Algorithm \ref{algorithm} terminates, whereas $\beta_1(\pi)$ is the initial letter of that valid factor and is the letter being considered by the algorithm at that point. Since the letter being considered at any point is the largest letter of the valid factor (or is not in the valid factor), it follows that $y<\beta_1(\pi)$.
\end{proof}
 
\section{Valley-hopping and cyclic valley-hopping}

Most of our remaining results will deal with the interplay between rixed points and valley-hopping; here we shall define the latter. Fix $\pi\in\mathfrak{S}_n$ and $x\in[n]$. We may write $\pi=w_1 w_2 x w_4 w_5$ where $w_2$ is the maximal consecutive subword immediately to the left of $x$ whose letters are all smaller than $x$, and $w_4$ is the maximal consecutive subword immediately to the right of $x$ whose letters are all smaller than $x$. Define $\varphi_x\colon\mathfrak{S}_n\rightarrow\mathfrak{S}_n$ by 
\[
\varphi_x(\pi)\coloneqq\begin{cases}
w_1 w_4 x w_2 w_5, & \text{if }x\text{ is a double ascent or double descent of }\pi,\\
\pi, & \text{if }x\text{ is a peak or valley of }\pi.
\end{cases}
\]

It is easy to see that $\varphi_x$ is an involution, and that $\varphi_x$ commutes with $\varphi_y$ for all $x,y\in[n]$. Thus, given a subset $S\subseteq[n]$, it makes sense to define $\varphi_S\colon\mathfrak{S}_n\rightarrow\mathfrak{S}_n$ by $\varphi_S\coloneqq\prod_{x\in S}\varphi_x$. The involutions $\{\varphi_S\}_{S\subseteq[n]}$ define a $\mathbb{Z}_2^n$-action on $\mathfrak{S}_n$, called the \textit{modified Foata--Strehl action} or \textit{valley-hopping}.

For example, if $\pi=834279156$ and $S=\{6,7,8\}$, then we have $\varphi_{S}(\pi)=734289615$; see Figure~\ref{f-vh}. This figure makes it apparent that, pictorially, the elements of $S$ are indeed ``hopping'' over valleys upon applying $\varphi_{S}$. Also, observe that $x\in S$ is a double ascent of $\pi$ if and only if $x$ is a double descent of $\varphi_{S}(\pi)$, and that it is a double descent of $\pi$ if and only if it is a double ascent of $\varphi_{S}(\pi)$.

\begin{figure}[t]
\begin{center}
\begin{tikzpicture}[scale=0.5] 	
\draw[step=1,lightgray,thin] (0,1) grid (10,10); 
	\tikzstyle{ridge}=[draw, line width=1, dotted, color=black] 
	\path[ridge] (0,10)--(1,8)--(2,3)--(3,4)--(4,2)--(5,7)--(6,9)--(7,1)--(8,5)--(9,6)--(10,10); 
	\tikzstyle{node0}=[circle, inner sep=2, fill=black] 
	\tikzstyle{node1}=[rectangle, inner sep=3, fill=dartmouthgreen] 
	\tikzstyle{node2}=[diamond, inner sep=2, fill=davidsonred] 
	\node[node2] at (1,8) {}; 
	\node[node0] at (2,3) {}; 
	\node[node0] at (3,4) {}; 
	\node[node0] at (4,2) {}; 
	\node[node2] at (5,7) {}; 
	\node[node0] at (6,9) {}; 
	\node[node0] at (7,1) {}; 
	\node[node0] at (8,5) {}; 
	\node[node2] at (9,6) {}; 
	\tikzstyle{hop1}=[draw, line width = 1.5, color=davidsonred,->]
	\tikzstyle{hop2}=[draw, line width = 1.5, color=davidsonred,<-] 
	\path[hop1] (1.5,8)--(5,8);
	\path[hop2] (1.6,7)--(4.5,7);
	\path[hop2] (6.8,6)--(8.5,6); 
	\tikzstyle{pi}=[above=-5] 
	\node[pi] at (0,0) {$\infty$}; 
	\node[pi, color=davidsonred] at (1,0) {8}; 
	\node[pi] at (2,0) {3}; 
	\node[pi] at (3,0) {4}; 
	\node[pi] at (4,0) {2}; 
	\node[pi, color=davidsonred] at (5,0) {7}; 
	\node[pi] at (6,0) {9}; 
	\node[pi] at (7,0) {1}; 
	\node[pi] at (8,0) {5}; 
	\node[pi, color=davidsonred] at (9,0) {6}; 
	\node[pi] at (10,0) {$\infty$}; 
 
	\path[draw,line width=1,->] (11,5)--(15,5); 
    \node at (13,5.5) {$\varphi_S$};
    
	\begin{scope}[shift={(16,0)}] 
	\draw[step=1,lightgray,thin] (0,1) grid (10,10); 
	\path[ridge] (0,10)--(1,7)--(2,3)--(3,4)--(4,2)--(5,8)--(6,9)--(7,6)--(8,1)--(9,5)--(10,10); 
	\node[node0] at (1,7) {}; 
	\node[node0] at (2,3) {}; 
	\node[node0] at (3,4) {}; 
	\node[node0] at (4,2) {}; 
	\node[node0] at (5,8) {}; 
	\node[node0] at (6,9) {}; 
	\node[node0] at (7,6) {}; 
	\node[node0] at (8,1) {}; 
	\node[node0] at (9,5) {}; 
	\node[pi] at (0,0) {$\infty$}; 
	\node[pi] at (1,0) {7}; 
	\node[pi] at (2,0) {3}; 
	\node[pi] at (3,0) {4}; 
	\node[pi] at (4,0) {2}; 
	\node[pi] at (5,0) {8}; 
	\node[pi] at (6,0) {9};
	\node[pi] at (7,0) {6}; 
	\node[pi] at (8,0) {1}; 
	\node[pi] at (9,0) {5}; 
	\node[pi] at (10,0) {$\infty$}; 
	\end{scope}
\end{tikzpicture}
\end{center}

\caption{Valley-hopping on $\pi=834279156$ with $S=\{6,7,8\}$ yields $\varphi_{S}(\pi)=734289615$.}\label{f-vh}
\end{figure}
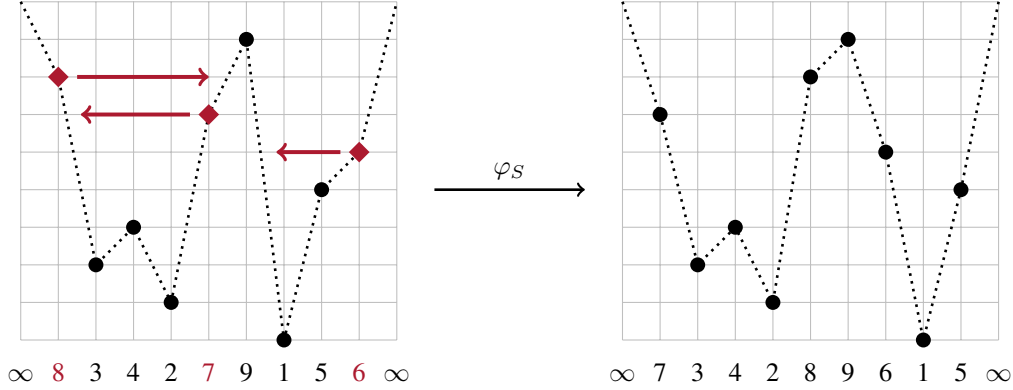

The valley-hopping action originated in work of Foata and Strehl \cite{Foata1974}, and was independently discovered by Shapiro, Woan, and Getu \cite{Shapiro1983} and by Br{\"a}nd{\'e}n \cite{Braenden2008}. A cyclic version of valley-hopping was later defined by Sun and Wang \cite{Sun2014} for derangements, and then extended to the entire symmetric group by Cooper, Jones, and the second author \cite{Cooper2020}. Below, we will extend Sun and Wang's action in a slightly different way.

As of this point, we have only needed to write permutations in one-line notation, but we shall now need both one-line notation and cycle notation. When writing a permutation in cycle notation, we shall write each cycle with its largest value first and listing the cycles in increasing order of their largest value; this convention is referred to as \textit{canonical cycle representation}.

Let $o\colon\mathfrak{S}_n\rightarrow\mathfrak{S}_n$ denote Foata's \textit{fundamental transformation}, which takes a permutation $\pi$ in canonical cycle representation and outputs the permutation $o(\pi)$ in one-line notation obtained from $\pi$ by erasing the parentheses. Given $x\in[n]$ and $S\subseteq[n]$, define $\psi_x\colon\mathfrak{S}_n\rightarrow\mathfrak{S}_n$ by $\psi_x\coloneqq o^{-1}\circ\varphi_x\circ o$ and $\psi_S\colon\mathfrak{S}_n\rightarrow\mathfrak{S}_n$ by $\psi_S\coloneqq\prod_{x\in S}\psi_x$. The $\{\psi_S\}_{S\subseteq[n]}$ induce a $\mathbb{Z}_2^n$-action on $\mathfrak{S}_n$ which we call \textit{cyclic valley-hopping}. See Figure \ref{f-cvh} for an example.

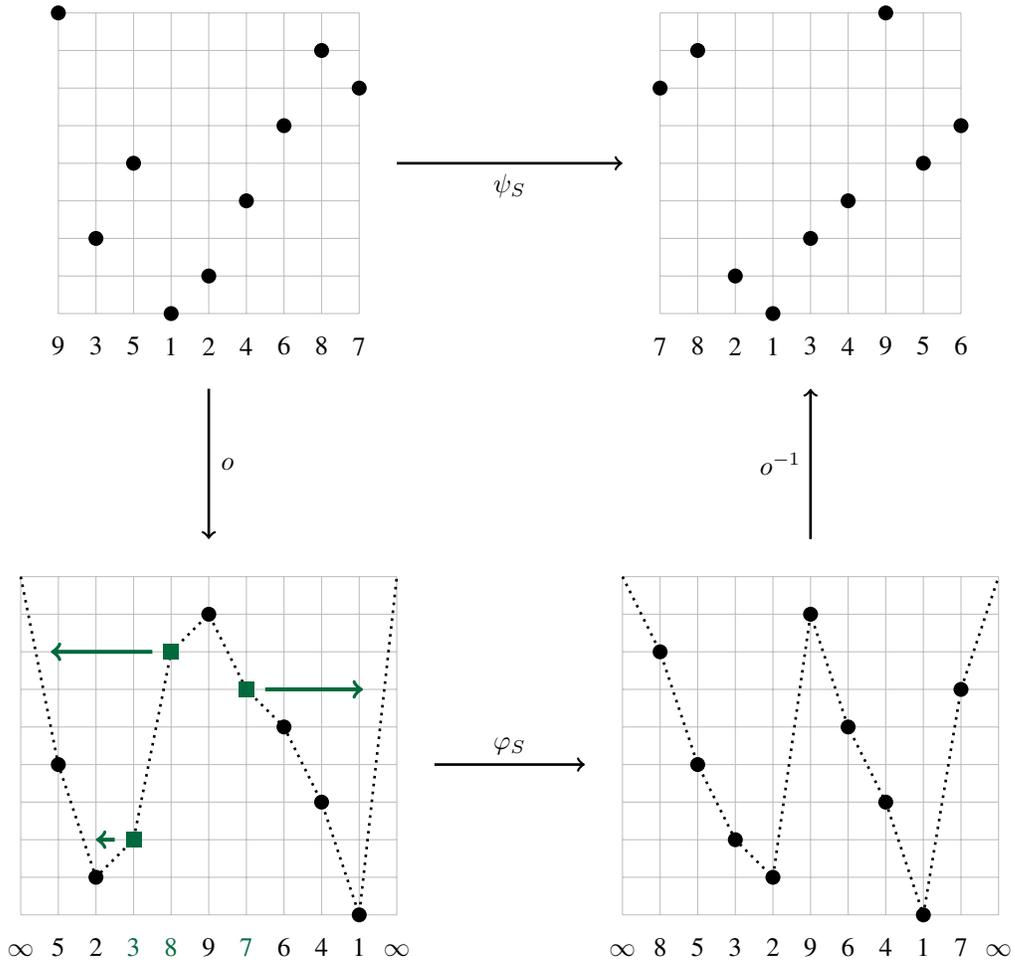
\begin{figure}[!t]
\begin{center}
\begin{tikzpicture}[scale=0.5] 	
\draw[step=1,lightgray,thin] (0,1) grid (10,10); 
	\tikzstyle{ridge}=[draw, line width=1, dotted, color=black] 
	\path[ridge] (0,10)--(1,5)--(2,2)--(3,3)--(4,8)--(5,9)--(6,7)--(7,6)--(8,4)--(9,1)--(10,10); 
	\tikzstyle{node0}=[circle, inner sep=2, fill=black] 
	\tikzstyle{node1}=[rectangle, inner sep=3, fill=dartmouthgreen] 
	\tikzstyle{node2}=[diamond, inner sep=2, fill=davidsonred] 
	\node[node0] at (1,5) {}; 
	\node[node0] at (2,2) {}; 
	\node[node1] at (3,3) {}; 
	\node[node1] at (4,8) {}; 
	\node[node0] at (5,9) {}; 
	\node[node1] at (6,7) {}; 
	\node[node0] at (7,6) {}; 
	\node[node0] at (8,4) {}; 
	\node[node0] at (9,1) {}; 
	\tikzstyle{hop1}=[draw, line width = 1.5, color=dartmouthgreen,->]
	\tikzstyle{hop2}=[draw, line width = 1.5, color=dartmouthgreen,<-] 
	\path[hop2] (2,3)--(2.5,3);
    \path[hop2] (0.8,8)--(3.5,8);
	\path[hop1] (6.5,7)--(9.1,7);
	\tikzstyle{pi}=[above=-5] 
	\node[pi] at (0,0) {$\infty$}; 
	\node[pi] at (1,0) {5}; 
	\node[pi] at (2,0) {2}; 
	\node[pi, color=dartmouthgreen] at (3,0) {3}; 
	\node[pi, color=dartmouthgreen] at (4,0) {8}; 
	\node[pi] at (5,0) {9}; 
	\node[pi, color=dartmouthgreen] at (6,0) {7}; 
	\node[pi] at (7,0) {6}; 
	\node[pi] at (8,0) {4}; 
	\node[pi] at (9,0) {1}; 
	\node[pi] at (10,0) {$\infty$}; 
	
	\path[draw,line width=1,->] (11,5)--(15,5);
	\node at (13,5.5) {$\varphi_S$};
	
	\begin{scope}[shift={(16,0)}] 
	\draw[step=1,lightgray,thin] (0,1) grid (10,10); 
	\path[ridge] (0,10)--(1,8)--(2,5)--(3,3)--(4,2)--(5,9)--(6,6)--(7,4)--(8,1)--(9,7)--(10,10); 
	\node[node0] at (1,8) {}; 
	\node[node0] at (2,5) {}; 
	\node[node0] at (3,3) {}; 
	\node[node0] at (4,2) {}; 
	\node[node0] at (5,9) {}; 
	\node[node0] at (6,6) {}; 
	\node[node0] at (7,4) {}; 
	\node[node0] at (8,1) {}; 
	\node[node0] at (9,7) {}; 
	\node[pi] at (0,0) {$\infty$}; 
	\node[pi] at (1,0) {8}; 
	\node[pi] at (2,0) {5}; 
	\node[pi] at (3,0) {3}; 
	\node[pi] at (4,0) {2}; 
	\node[pi] at (5,0) {9}; 
	\node[pi] at (6,0) {6};
	\node[pi] at (7,0) {4}; 
	\node[pi] at (8,0) {1}; 
	\node[pi] at (9,0) {7}; 
	\node[pi] at (10,0) {$\infty$}; 
	\end{scope}
	
	\path[draw,line width=1,->] (5,15)--(5,11);
	\node at (5.5,13) {$o$};
	\path[draw,line width=1,->] (21,11)--(21,15);
	\node at (20.2,13) {$o^{-1}$};
	
	\begin{scope}[shift={(0,16)}] 
	\draw[step=1,lightgray,thin] (1,1) grid (9,9); 
	\node[node0] at (1,9) {}; 
	\node[node0] at (2,3) {}; 
	\node[node0] at (3,5) {}; 
	\node[node0] at (4,1) {}; 
	\node[node0] at (5,2) {}; 
	\node[node0] at (6,4) {}; 
	\node[node0] at (7,6) {}; 
	\node[node0] at (8,8) {}; 
	\node[node0] at (9,7) {}; 
	\node[pi] at (1,0) {9}; 
	\node[pi] at (2,0) {3}; 
	\node[pi] at (3,0) {5}; 
	\node[pi] at (4,0) {1}; 
	\node[pi] at (5,0) {2}; 
	\node[pi] at (6,0) {4};
	\node[pi] at (7,0) {6}; 
	\node[pi] at (8,0) {8}; 
	\node[pi] at (9,0) {7}; 
	\end{scope}
	
	\path[draw,line width=1,->] (10,21)--(16,21);
	\node at (13,20.4) {$\psi_S$};
	
	\begin{scope}[shift={(16,16)}] 
	\draw[step=1,lightgray,thin] (1,1) grid (9,9); 
	\node[node0] at (1,7) {}; 
	\node[node0] at (2,8) {}; 
	\node[node0] at (3,2) {}; 
	\node[node0] at (4,1) {}; 
	\node[node0] at (5,3) {}; 
	\node[node0] at (6,4) {}; 
	\node[node0] at (7,9) {}; 
	\node[node0] at (8,5) {}; 
	\node[node0] at (9,6) {}; 
	\node[pi] at (1,0) {7}; 
	\node[pi] at (2,0) {8}; 
	\node[pi] at (3,0) {2}; 
	\node[pi] at (4,0) {1}; 
	\node[pi] at (5,0) {3}; 
	\node[pi] at (6,0) {4};
	\node[pi] at (7,0) {9}; 
	\node[pi] at (8,0) {5}; 
	\node[pi] at (9,0) {6}; 
	\end{scope}
\end{tikzpicture}
\end{center}

\caption{Cyclic valley-hopping on $\pi=(523)(8)(97641)=935124687$ with $S=\{3,7,8\}$
yields $\psi_{S}(\pi)=(8532)(96417)=782134956$.} \label{f-cvh}
\end{figure}

We will also consider ``restricted'' versions of valley-hopping and cyclic valley-hopping. Define \textit{restricted valley-hopping} to be the $\mathbb{Z}_2^n$-action on $\mathfrak{S}_n$ induced by the involutions $\hat{\varphi}_S\coloneqq\prod_{x\in S}\hat{\varphi}_x$ where
\[
\hat{\varphi}_x(\pi)\coloneqq\begin{cases}
\pi, & \text{if }x\in\Rix(\pi)\text{ or if }x=\beta_1(\pi),\\
\varphi_x(\pi), & \text{otherwise}.
\end{cases}
\]
Moreover, define \textit{restricted cyclic valley-hopping} to be the $\mathbb{Z}_2^n$-action on $\mathfrak{S}_n$ induced by the involutions $\hat{\psi}_S\coloneqq\prod_{x\in S}\hat{\psi}_x$ where
\[
\hat{\psi}_x(\pi)\coloneqq\begin{cases}
\pi, & \text{if }x\in\Fix(\pi)\text{ or if }x\text{ is the first letter of }o(\pi),\\
\psi_x(\pi), & \text{otherwise}.
\end{cases}
\]
Restricted valley-hopping was first defined by Lin and Zeng \cite[Section 4]{Lin2015}, and restricted cyclic valley-hopping is precisely the aforementioned extension of Sun and Wang's action due to Cooper, Jones, and the second author.

\section{\texorpdfstring{$\rix$}{rix} is homomesic under valley-hopping}

Having defined the valley-hopping action, our next goal is to prove the following.

\begin{thm} \label{t-rixmesic}
The $\rix$ statistic is 1-mesic under valley-hopping.
\end{thm}

\noindent A statistic is \textit{$k$-mesic} under an action if its average value over each orbit is equal to $k$. In other words, we claim that the permutations in each valley-hopping orbit have 1 rixed point on average.

Given an orbit $\Pi$ of the valley-hopping action, define the set $R_{\Pi}$ by
\[
R_{\Pi}\coloneqq\{\,(\pi,x):\pi\in\Pi\text{ and }x\in\Rix(\pi)\,\}
\]
and the map $\phi\colon R_{\Pi}\rightarrow\Pi$ by taking $\phi(\pi,x) \coloneqq \varphi_x (\pi)$---i.e., the permutation in $\Pi$ obtained by letting $x$ hop in $\pi$.

\begin{lem} 
\label{l-beta1}
For any $(\pi,x)\in R_{\Pi}$, we have $x=\beta_1(\phi(\pi,x))$.
\end{lem}

\begin{proof}
Fix $(\pi,x)\in R_{\Pi}$ and let 
\[
\pi=\alpha_1\cdots\alpha_k\beta\qquad\text{and}\qquad\phi(\pi,x)=\alpha_1^{\prime}\cdots\alpha_m^{\prime}\beta^{\prime}
\]
be the rix-factorizations of $\pi$ and $\phi(\pi,x)$, respectively. Suppose that $x=\beta_1(\pi)$. Then $x$ is a valley of $\pi$ by Lemma \ref{l-valdasc}, so $\pi=\phi(\pi,x)$ and thus $x=\beta_1(\pi)=\beta_1(\phi(\pi,x))$.
Hence, let us assume for the rest of this proof that $x\neq\beta_1(\pi)$, which by Lemma \ref{l-valdasc} means that $x$ is a double ascent of $\pi$.

Since $x$ is a double ascent of $\pi$, we know that $x$ is a double descent of $\phi(\pi,x)$. In addition, we know that $x$ hops over $\beta_1(\pi)$ because $x>\beta_1(\pi)$---that is, $x$ appears before $\beta_1(\pi)$ in $\phi(\pi,x)$. In fact, we claim that $x$ is a leading descent of $\phi(\pi,x)$. To see this, first recall from Definition \ref{defn:rix_lin_zeng} that either $\beta$ is increasing or $\beta_1(\pi)$ is the largest descent of $\beta$; in either case, there cannot be a descent of $\phi(\pi,x)$ larger than $x$ to the right of $\beta_1(\pi)$. There also cannot be a descent of $\phi(\pi,x)$ larger than $x$ located between $x$ and $\beta_1(\pi)$, as $x$ would not have been able to hop over that descent. Therefore, $x$ is a leading descent of $\phi(\pi,x)$, which by Lemma \ref{l-leaddes} implies that $x$ is either the last letter of some $\alpha_i^{\prime}$ or belongs to $\beta^{\prime}$. 

Assume by way of contradiction that $x$ is the last letter of $\alpha_i^{\prime}$. If $x$ is the first letter of $\phi(\pi,x)$, then $x$ cannot be the last letter of $\alpha_i^{\prime}$ because $\alpha_i^{\prime}$ has length at least 2 by Definition \ref{defn:rix_lin_zeng}; but in this case, we would have $\beta^{\prime}=\phi(\pi,x)$ which gives the desired conclusion. Otherwise, let $y$ be the letter immediately preceding $x$ in $\phi(\pi,x)$. We know that $y>x$ by the definition of valley-hopping; after all, if $y<x$, then $x$ would have hopped over it. In other words, $y$ is a descent of $\phi(\pi,x)$. In fact, $y$ is also a descent of $\pi$; the letter $z$ immediately following $y$ in $\pi$ appears after $x$ in $\phi(\pi,x)$, and so $x>z$ (and thus $y>z$) because $x$ hopped over it. The following illustrates the relative placements of $y$, $z$, $\beta_1(\pi)$, $x$ before and after valley-hopping:
\[
\pi=\cdots yz\cdots\beta_1(\pi)\cdots x\cdots\qquad\phi(x,\pi)=\cdots yxz\cdots\beta_1(\pi)\cdots.
\]
Because $x$ is a leading descent of $\phi(\pi,x)$ and $y>x$, it follows that $y$ is a leading descent of both $\pi$ and $\phi(x,\pi)$. We know that $y$ cannot be in $\beta^{\prime}$ because $x$ is not in $\beta^{\prime}$, so by Lemma \ref{l-leaddes}, it must be true that $y$ is the last letter of $\alpha_{i-1}^{\prime}$, and yet this is impossible because it would mean that $\alpha_i^{\prime}$ has length 1. Therefore, our assumption that $x$ is the last letter of some $\alpha_i^{\prime}$ is false, and so $x$ belongs to $\beta^{\prime}$ and either $y$ is the last letter of $\alpha_m^{\prime}$ or also belongs to $\beta^{\prime}$. If we can show that $y$ is the last letter of $\alpha_m^{\prime}$, then $x$ would be the first letter of $\beta^{\prime}$ as desired.

Up to this point, we've only used the fact that $y$ is a leading descent of $\phi(x,\pi)$, but recall that $y$ is also a leading descent of $\pi$. Since $y$ is not in $\beta$, by appealing to Lemma \ref{l-leaddes} again, it follows that $y$ is the last letter of some $\alpha_j$. Letting $x$ hop does not change the prefix of the permutation $\pi$ up to and including $y$, nor does it change whether $y$ is a leading descent, so the first $j$ rix-factors $\alpha_1,\dots,\alpha_j$ of $\pi$ are precisely the rix-factors $\alpha_1^{\prime},\dots,\alpha_m^{\prime}$ of $\phi(x,\pi)$. Hence, $y$ is the last letter of $\alpha_m^{\prime}$ and we are done.
\end{proof}

\begin{prop} \label{p-bij}
The map $\phi\colon R_{\Pi}\rightarrow\Pi$ is a bijection.
\end{prop}

\begin{proof}
In light of Lemma \ref{l-beta1}, we can recover $x$ from $\phi(\pi,x)$ by taking $x=\beta_1(\phi(\pi,x))$, and we can recover $\pi$ from $\phi(\pi,x)$ and $x$ by letting $x$ hop in $\phi(\pi,x)$.
\end{proof}

Theorem \ref{t-rixmesic} is an immediate corollary of Proposition \ref{p-bij}. After all, the fact that $\phi$ is a bijection tells us that the total number of rixed points among permutations in $\Pi$ is equal to the number of permutations in $\Pi$. In other words, the average value of $\rix$ over any valley-hopping orbit is 1.

\section{\texorpdfstring{$\Phi$}{Phi} sends valley-hopping orbits to cyclic valley-hopping orbits}
As mentioned in the introduction, Lin and Zeng \cite{Lin2015} define a bijection $\Phi\colon\mathfrak{S}_n\rightarrow\mathfrak{S}_n$ satisfying $\des(\pi)=\exc(\Phi(\pi))$ and $\Rix(\pi)=\Fix(\Phi(\pi))$. The remainder of our paper will be devoted to proving the following theorem.

\begin{thm} 
\label{t-main}
Let $\pi$ be a permutation, $\Pi$ the valley-hopping orbit containing $\pi$, $\hat{\Pi}$ the restricted valley-hopping orbit containing $\pi$, $\Pi^{\prime}$ the cyclic valley-hopping orbit containing $\Phi(\pi)$, and $\hat{\Pi}^{\prime}$ the restricted cyclic valley-hopping orbit containing $\Phi(\pi)$. Then:
\begin{enumerate}
\item[\normalfont(a)]  $\Phi(\Pi)=\Pi^{\prime}$
\item[\normalfont(b)]  $\Phi(\hat{\Pi})=\hat{\Pi}^{\prime}$
\end{enumerate}
\end{thm}

In other words, $\Phi$ sends orbits of the valley-hopping action to orbits of cyclic valley-hopping---so that these actions are in sense ``the same'' up to $\Phi$---and the restricted versions of valley-hopping and cyclic valley-hopping are related in the same way. Before reviewing the definition of $\Phi$ and working toward the proof of Theorem \ref{t-main}, we note that the following is an immediate consequence of Theorems \ref{t-rixmesic} and \ref{t-main} since $\Rix(\pi)=\Fix(\Phi(\pi))$ implies $\rix(\pi)=\fix(\Phi(\pi))$.

\begin{thm} \label{t-fixmesic}
The $\fix$ statistic is 1-mesic under cyclic valley-hopping.
\end{thm}

\begin{remark}
In \cite{LaCroixRoby} and \cite{Sheridan-Rossi}, the authors show that the number of fixed points is homomesic under some ``Foatic actions", which are compositions of the form $f \circ o^{-1} \circ g \circ o$ where $f,g$ are dihedral actions (such as the reverse map, the inverse map, and the complement map). We checked whether the rix statistic is homomesic under any Foatic actions, but found counterexamples for all of them.
\end{remark}

\subsection{The bijection \texorpdfstring{$\Phi$}{Phi}}

Let $\pi\in\mathfrak{S}_n$ have rix-factorization 
\[
\pi=\alpha_1 \alpha_2 \cdots \alpha_k \beta.
\]
Also, let $\Rix(\pi)=\{\pi_j,\pi_{j+1},\dots,\pi_n\}$ and let $\delta$ be defined by $\beta\coloneqq\delta\pi_j\pi_{j+1}\cdots\pi_n$---that is, $\delta$ is obtained from $\beta$ upon removing all rixed points. If $\delta=d_1 d_2 \cdots d_l$, then let
\[
\tilde{\delta}\coloneqq(d_1,d_l,d_{l-1},\dots,d_2)
\]
and for each $\alpha_i = a_1 a_2 \cdots a_l$, let
\[
\tilde{\alpha}_i\coloneqq(a_l,a_{l-1},\dots,a_1).
\]
Then Lin and Zeng define $\Phi(\pi)$ to be the following concatenation of cycles:
\begin{equation}
\Phi(\pi)\coloneqq\tilde{\alpha}_1\tilde{\alpha}_2\cdots\tilde{\alpha}_k\tilde{\delta}(\pi_j)(\pi_{j+1})\cdots(\pi_n).\label{e-Phidef}
\end{equation}
For example, given $\pi=7\,6\,9\,1\,8\,4\,2\,3\,5\,10\,11$, we have
\[
\Phi(\pi)=(9,6,7)(8,1)(4,3,2)(5)(10)(11).
\]

In (\ref{e-Phidef}), each cycle is written with its largest letter first, the cycles of length at least 2 are arranged in decreasing order of their largest letter,\footnote{If $x_i$ denotes the last letter of $\alpha_i$, then $x_1 > x_2 > \cdots > x_k > \beta_1(\pi)$ \cite[Proposition 9 (i)]{Lin2015}.} and the fixed points are arranged in increasing order and after the cycles of length at least 2. However, for our purposes, we will need to rearrange the order of the cycles to be in line with canonical cycle representation. To that end, let us write $\Phi(\pi)$ as
\begin{equation}
\Phi(\pi)=\tilde{\delta}\mu_k\tilde{\alpha}_k\mu_{k-1}\tilde{\alpha}_{k-1}\cdots\mu_1\tilde{\alpha}_1\mu_0\label{e-can}
\end{equation}
where each $\mu_i$ consists of all fixed points (in increasing order) which are greater than the first entry of the previous cycle and (if $i>0$) less than the first entry of the subsequent cycle. Continuing with the example $\pi=7\,6\,9\,1\,8\,4\,2\,3\,5\,10\,11$, the cycles of $\Phi(\pi)$ are rearranged to become
\[
\Phi(\pi)=(4,3,2)(5)(8,1)(9,6,7)(10)(11)
\]
so that $\mu_2=(5)$, $\mu_1$ is empty, and $\mu_0=(10)(11)$.

\subsection{Proof of Theorem \ref{t-main}}

Our proof of Theorem \ref{t-main} will require a few additional lemmas.
\begin{lem} \label{l-pkvaldbl}
Let $\pi\in\mathfrak{S}_n$ and $x\in[n]$. Then:
\begin{enumerate}
\item[\normalfont(a)]  $x$ is a peak of $\pi$ if and only if $x$ is a peak of $o(\Phi(\pi))$;
\item[\normalfont(b)]  $x$ is a valley of $\pi$ if and only if $x$ is a valley of $o(\Phi(\pi))$;
\item[\normalfont(c)]  $x$ is a double ascent or a double descent of $\pi$ if and only if $x$ is a double ascent or double descent of $o(\Phi(\pi))$.
\end{enumerate}
\end{lem}

Before giving the proof of Lemma \ref{l-pkvaldbl}, let us briefly describe the intuition behind this lemma. The map $o \circ \Phi$ takes a permutation $\pi$ in one-line notation, considers the permutation $\Phi(\pi)$ in cycle notation where the cycles are determined by the rix-factorization of $\pi$, but then erases the parentheses to obtain the permutation $o(\Phi(\pi))$ in one-line notation. The overarching idea of the proof is to show that if $x$ is a peak of $\pi$ and if $o \circ \Phi$ changes the neighboring letters of $\pi$, then $x$ will stay a peak, and that the same is true if $x$ is instead a valley. However, we will need to carefully check a number of cases to verify that this is indeed true.

\begin{proof}
We shall first establish the forward directions of (a) and (b): if $x$ is a peak (respectively, valley) of $\pi$, then $x$ is a peak (respectively, valley) of $o(\Phi(\pi))$.

\medskip{}
\noindent \textbf{Case 1:} $x$ is an element of $\alpha_i$ for some $1\leq i\leq k$.

Let us write $\alpha_i = a_1 \cdots a_j x a_{j+1} \cdots a_l$, so that in $\Phi(\pi)$ we have 
\[
\tilde{\alpha}_i=(a_l,\dots,a_{j+1},x,a_j,\dots,a_1).
\]
Suppose that $x$ is a peak of $\pi$. Because each $\alpha_i$ ends with a descent, $x$ cannot be the first letter of $\alpha_i$. This means that $a_1\cdots a_j$ cannot be empty but $a_{j+1}\cdots a_l$ can. If $a_{j+1}\cdots a_l$ is not empty, then $x$ is clearly a peak of $o(\Phi(\pi))$. If $a_{j+1}\cdots a_l$ is empty, then $x$ would be the first letter of the cycle $\tilde{\alpha}_i$, so as long as $\tilde{\alpha}_i$ is not the first cycle of $\Phi(\pi)$, we are guaranteed by canonical cycle representation that $x$ is a peak of $o(\Phi(\pi))$. If $\tilde{\alpha}_i$ were the first cycle of $\Phi(\pi)$, then that means $\delta$ is empty and that there are no rixed points of $\pi$ smaller than $x$; however, that would imply that $x$ is immediately followed in $\pi$ by a rixed point larger than $x$, which is a contradiction because $x$ is the last letter of $\alpha_i$ and thus a descent. Therefore, $x$ is a peak of $o(\Phi(\pi))$.

Now, suppose that $x$ is a valley of $\pi$. Because each $\alpha_i$ ends with a descent, $x$ cannot be the last letter of $\alpha_i$. This means that $a_1\cdots a_j$ can be empty but $a_{j+1}\cdots a_l$ cannot. Similar to above, $x$ is clearly a valley of $o(\Phi(\pi))$ if $a_1\cdots a_j$ is not empty. If $a_1\cdots a_j$ is empty, then $x$ is the last letter of the cycle $\tilde{\alpha}_i$, in which case $x$ would still be a valley of $o(\Phi(\pi))$ by the definition of canonical cycle representation.

\medskip{}
\noindent \textbf{Case 2:} $x$ is in $\beta$ but is neither $\beta_1(\pi)$
nor a rixed point of $\pi$.

Recall that $\delta=d_1d_2\cdots d_l$ is obtained from $\beta$ by deleting all the rixed points, and that 
\[
\tilde{\delta}=(d_1,d_l,d_{l-1},\dots,d_3,d_2).
\]
In this case, we have $x=d_i$ for some $i\neq1$, and it is easy to see that the desired result holds when $i\neq2$ and $i\neq l$. So it remains to check the cases when $x=d_2$ and $x=d_l$. Recall that $d_1>x$ (guaranteed by Lemma \ref{l-xbet}), and that $d_l$ is either the last letter of $\pi$ or is followed by a rixed point (which is by definition larger than $d_1$ and thus larger than $d_l$); hence, in none of these cases can $x$ be a peak of $\pi$. Then consider the following subcases:

\begin{itemize}
\item Suppose that $x=d_2=d_l$, so that $\delta=d_1x$ and $\tilde{\delta}=(d_1,x)$. Then $x$ is a valley in both $\pi$ and $o(\Phi(\pi))$.
\item Suppose that $x=d_2\neq d_l$, so that $\delta=d_1 x d_3 \cdots d_l$ and $\tilde{\delta}=(d_1,d_l,\dots,d_3,x)$. Then $x$ is a valley of $\pi$ if and only if $x<d_3$. As the last entry of the cycle $\tilde{\delta}$ in canonical cycle representation, $x$ is either followed by a larger letter in $o(\Phi(\pi))$ or is the last letter of $o(\Phi(\pi))$. Either way, we see that $x$ is also a valley of $o(\Phi(\pi))$ when $x<d_3$.
\item Suppose that $x=d_l\neq d_2$, so that $\delta=d_1 d_2 \cdots d_{l-1} x$ and $\tilde{\delta}=(d_1,x,d_{l-1},\dots,d_3)$. Then $x$ is a valley if and only if $x<d_{l-1}$, in which case it is also a valley of $o(\Phi(\pi))$.
\end{itemize}

\noindent \textbf{Case 3:} $x=\beta_1(\pi)$ is a rixed point of
$\pi$.

By Lemma \ref{l-valdasc}, we know that in this case $x$ is a valley of $\pi$. Note that $(x)$ is a fixed point of $\Phi(\pi)$ and is in fact the first cycle of $\Phi(\pi)$, so $x$ is the first letter of $o(\Phi(\pi))$. Per canonical cycle representation, $x$ is either followed by a larger letter in $o(\Phi(\pi))$ or is the last letter of $o(\Phi(\pi))$. Either way, $x$ is also a valley of $o(\Phi(\pi))$.

\medskip{}
The above three cases are the only ones that we need consider. Indeed, if $x$ is a rixed point of $\pi$ but is not $\beta_1(\pi)$, then $x$ is a double ascent of $\pi$ by Lemma \ref{l-valdasc}. Furthermore, if $x=\beta_1(\pi)$ is not a rixed point of $\pi$, then $x$ is a double descent of $\pi$ by Lemma \ref{l-xbet} and the fact that $x=\beta_1(\pi)$ is either the first letter of $\pi$ or is preceded by a peak. Hence, the forward directions of (a) and (b) follow.

Now, note that the forward direction of (a) implies that $\pk(\pi)\leq\pk(o(\Phi(\pi)))$ for all $\pi\in\mathfrak{S}_n$, where $\pk(\pi)$ denotes the number of peaks of $\pi$. Moreover, the $o(\Phi(\pi))$ span all permutations in $\mathfrak{S}_n$ because $o$ and $\Phi$ are bijections; so, if it were not true that $\pk(\pi)=\pk(o(\Phi(\pi)))$ for all $\pi\in\mathfrak{S}_n$, then summing over all $\pi\in\mathfrak{S}_n$ would result in the absurdity that the total number of peaks over all $\pi\in\mathfrak{S}_n$ is less than the total number of peaks over all $\pi\in\mathfrak{S}_n$. Hence, the backward direction of (a) is established, and the backward direction of (b) follows from the same reasoning.

Finally, it is clear that (a) and (b) imply (c), and thus the proof is complete.
\end{proof}

\begin{cor} \label{c-orbsize} 
Let $\pi$ be a permutation, $\Pi$ the valley-hopping orbit containing $\pi$, $\hat{\Pi}$ the restricted valley-hopping orbit containing $\pi$, $\Pi^{\prime}$ the cyclic valley-hopping orbit containing $\Phi(\pi)$, and $\hat{\Pi}^{\prime}$ the restricted cyclic valley-hopping orbit containing $\Phi(\pi)$. Then:
\begin{enumerate}
\item[\normalfont(a)]  $|\Pi|=|\Pi^{\prime}|$
\item[\normalfont(b)]  $|\hat{\Pi}|=|\hat{\Pi}^{\prime}|$
\end{enumerate}
\end{cor}

\begin{proof}
Let $\dbl(\pi)$ denote the total number of double ascents and double descents of $\pi$. Then the number of permutations in $\Pi$ is equal to $2^{\dbl(\pi)}$. Similarly, the number of permutations in $\Pi^{\prime}$ is $2^{\dbl(o(\Phi(\pi)))}$. Lemma \ref{l-pkvaldbl} (c) implies $\dbl(\pi)=\dbl(o(\Phi(\pi)))$, which completes the proof of (a). 

To prove (b), let us first make the following observations. First, it is clear from the definition of canonical cycle representation that every fixed point of $\Phi(\pi)$ is a double ascent of $o(\Phi(\pi))$ unless the fixed point is the first cycle of $\Phi(\pi)$, in which case it is a valley of $o(\Phi(\pi))$. In addition, it is easy to see that the first cycle of $\Phi(\pi)$ is a fixed point if and only if the first cycle is $(\beta_1(\pi))$, which occurs if and only if $\beta_1(\pi)$ is a rixed point of $\pi$. As such, let us divide into the following cases:

\medskip{}
\noindent \textbf{Case 1:} $\beta_1(\pi)$ is a rixed point of $\pi$. 

By Lemma \ref{l-valdasc}, we know that $\beta_1(\pi)$ is a valley of $\pi$ while all of the other rixed points are double ascents of $\pi$, so the number of permutations in $\hat{\Pi}$ is $2^{\dbl(\pi)-\rix(\pi)+1}$. On the other hand, from the discussion above, we know that the number of permutations in $\hat{\Pi}^{\prime}$ is $2^{\dbl(o(\Phi(\pi)))-\fix(\Phi(\pi))+1}=2^{\dbl(\pi)-\rix(\pi)+1}$.

\medskip{}
\noindent \textbf{Case 2:} $\beta_1(\pi)$ is not a rixed point of $\pi$. 

Appealing to Lemma \ref{l-valdasc} again, we see that all of the rixed points of $\pi$ are double ascents of $\pi$, so the number of permutations in $\hat{\Pi}$ is $2^{\dbl(\pi)-\rix(\pi)}$. Accordingly, the number of permutations in $\hat{\Pi}^{\prime}$ is $2^{\dbl(o(\Phi(\pi)))-\fix(\Phi(\pi))}=2^{\dbl(\pi)-\rix(\pi)}$.

\medskip{}
\noindent Thus the proof of (b) is complete.
\end{proof}

We now seek to show that whenever two permutations $\pi$ and $\sigma$ are in the same (restricted) valley-hopping orbit, then $\Phi(\pi)$ and $\Phi(\sigma)$ are in the same (restricted) cyclic valley-hopping orbit. Toward this goal, we prove the following lemma, which will again require extensive casework.

\begin{lem} \label{l-Phiclosed}
Let $\sigma=\varphi_x(\pi)$ and let $\pi=\alpha_1\alpha_2\cdots\alpha_k\beta$ be the rix-factorization of $\pi$.
\begin{enumerate}
\item[\normalfont(a)]  If $x$ is neither $\beta_1(\pi)$ nor a rixed point of $\pi$, then $\Phi(\sigma)=\psi_x(\Phi(\pi))$.
\item[\normalfont(b)]  If $x$ is a rixed point of $\pi$, then $\Phi(\sigma)=\psi_S(\Phi(\pi))$ where $S=\{\beta_1(\pi)\}\cup\{\,y\in\Rix(\pi):y\leq x\,\}$.
\item[\normalfont(c)]  If $x=\beta_1(\pi)$, then $\Phi(\sigma)=\psi_S(\Phi(\pi))$
where $S=\{\beta_1(\sigma)\}\cup\{\,y\in\Rix(\sigma):y\leq x\,\}$.
\end{enumerate}
\end{lem}

\begin{proof}
We divide into cases based on the position of $x$ in $\pi$. In all of the cases below, let 
\[
\sigma=\alpha_1^{\prime}\alpha_2^{\prime}\cdots\alpha_m^{\prime}\beta^{\prime}
\]
be the rix-factorization of $\sigma$. Cases 1--2 will establish part (a), whereas Cases 3--5 will prove parts (b) and (c).

\medskip{}
\noindent \textbf{Case 1:} $x$ is in $\alpha_i$ for some $1\leq i\leq k$. 

If $x$ is a peak or valley of $\pi$, then $\sigma=\pi$ and $x$ is also a peak or valley of $o(\Phi(\pi))$ by Lemma \ref{l-pkvaldbl}, which together imply $\Phi(\sigma)=\psi_x(\Phi(\pi))$. So, for the remainder of this case, let us assume that $x$ is neither a peak nor a valley of $\pi$. In particular, this means that $x$ is not the last letter of $\alpha_i$, since we know from Algorithm \ref{algorithm} that the last letter of each $\alpha_i$ is a peak of $\pi$. Observe that both the last letter of $\alpha_{i-1}$ (if $i>1$) and the last letter of $\alpha_i$ are larger than $x$; if $x$ were instead larger than either, then $x$ would have been considered by Algorithm \ref{algorithm} and thus removed from the valid factor prior to when $\alpha_i$ was added to the rix-factorization, which is impossible. Hence, upon applying $\varphi_x$ to $\pi$, the letter $x$ belongs to the same rix-factor, so the number of $\rix$-factors is unchanged. This means that $\pi$ and $\sigma$ have exactly the same $\rix$-factors except for $\alpha_i$ and $\alpha_i^{\prime}$, and similarly with the cycles in the canonical cycle representation of $\Phi(\pi)$ and $\Phi(\sigma)$.

Now, let us write
\[
\alpha_i=a_1 \cdots a_p x a_{p+1} \cdots a_l,\quad\text{so that}\quad\tilde{\alpha}_i=(a_l,\dots,a_{p+1},x,a_p,\dots,a_1).
\]
If $x$ is a double descent of $\pi$, then we have 
\[
\alpha_i^{\prime}=a_1 \cdots a_p a_{p+1}\cdots a_q x a_{q+1} \cdots a_l\quad\text{and}\quad\tilde{\alpha}_i^{\prime}=(a_l,\dots,a_{q+1},x,a_q,\dots,a_{p+1},a_p,\dots,a_1);
\]
here, $a_{q+1}$ is the closest letter to the right of $x$ in $\alpha_i$ that is larger than $x$. It is clear that when we apply $\psi_x$ to $\Phi(\pi)$, the cycle $\tilde{\alpha}_i$ is transformed to $\tilde{\alpha}_i^{\prime}$ and therefore $\Phi(\sigma)=\psi_x(\Phi(\pi))$. The case when $x$ is a double ascent of $\pi$ is similar.

\medskip{}
\noindent \textbf{Case 2:} $x$ is in $\beta$ but is neither $\beta_1(\pi)$ nor a rixed point of $\pi$.

By Lemma \ref{l-xbet}, we have $x<\beta_1(\pi)$, and therefore $x$ is also smaller than the rixed points of $\pi$. This means that $\Phi(\pi)$ and $\Phi(\sigma)$ have exactly the same cycles except for $\tilde{\delta}$ and $\tilde{\delta}^{\prime}$, and the remainder of the proof for this case follows in a similar way as in Case 1.

\medskip{}
\noindent \textbf{Case 3:} $x=\beta_1(\pi)$ is a rixed point of $\pi$.

We know from Lemma \ref{l-xbet} that, in this case, $x$ is the smallest rixed point of $\pi$. Moreover, by Lemma \ref{l-beta1}, we have $\beta_1(\sigma)=x$ and thus there are no rixed points of $\sigma$ smaller than $x$. This means that our set $S$ as defined in the statements of (b) and (c) is given by $S=\{x\}$. Note that $x$ is a valley of $\pi$ by Lemma \ref{l-valdasc} and therefore a valley of $o(\Phi(\pi))$ by Lemma \ref{l-pkvaldbl}. Since $x$ being a valley of $\pi$ implies $\sigma=\pi$, we have $\Phi(\sigma)=\psi_S(\Phi(\pi))$
as desired.

\medskip{}
\noindent \textbf{Case 4:} $x\neq\beta_1(\pi)$ is a rixed point of $\pi$.

Let 
\[
\Rix(\pi)=\{\pi_j,\pi_{j+1},\dots,\pi_{l-1},\pi_l=x,\pi_{l+1},\dots,\pi_n\}
\]
be the set of rixed points of $\pi$. Then we can write the rix-factorization of $\pi$ as
\[
\pi=\alpha_1\alpha_2\cdots\alpha_k\underset{\beta}{\underbrace{\delta\pi_j\cdots\pi_{l-1}x\pi_{l+1}\cdots\pi_n}}.
\]
As usual, we write
\[
\Phi(\pi)=\tilde{\delta}\mu_k\tilde{\alpha}_k\mu_{k-1}\tilde{\alpha}_{k-1}\cdots\mu_1\tilde{\alpha}_1\mu_0.
\]

From Lemma \ref{l-beta1}, we have that $x=\beta_1(\sigma)$. This means that the rix-factorization of $\sigma$ is given by 
\[
\sigma=\alpha_1\alpha_2\cdots\alpha_m\underset{\beta^{\prime}}{\underbrace{x\alpha_{m+1}\cdots\alpha_k\delta\pi_j\cdots\pi_{l-1}\pi_{l+1}\cdots\pi_n}.}
\]
Note that the last letter of $\alpha_m$ is the closest letter to the left of $x$ in $\pi$ that is larger than $x$, and that $\pi_j,\pi_{j+1},\dots,\pi_{l-1}$ are not rixed points of $\sigma$ because they are less than $x$. Taking $\delta^{\prime}$ to be the analogue of $\delta$ but for $\sigma$, we have 
\[
\delta^{\prime}=x\alpha_{m+1}\cdots\alpha_k\delta\pi_j\pi_{j+1}\cdots\pi_{l-1}
\]
and
\[
\Phi(\sigma)=\tilde{\delta}^{\prime}\mu_m^{\prime}\tilde{\alpha}_m^{\prime}\mu_{m-1}^{\prime}\tilde{\alpha}_{m-1}^{\prime}\cdots\mu_1^{\prime}\tilde{\alpha}_1^{\prime}\mu_0^{\prime}
\]
where $\tilde{\alpha}_i^{\prime}=\tilde{\alpha}_i$ for each $1\leq i\leq m$, $\mu_i^{\prime}=\mu_i$ for each $1\leq i\leq m-1$, and $\mu_m^{\prime}$ is obtained from $\mu_m$ by removing all the rixed points of $\pi$ that are less than or equal to $x$. 

Next, we characterize the cycle $\tilde{\delta}^{\prime}$. Assume that $\delta$ is nonempty; we omit the proof of the case where $\delta$ is empty as it is similar but slightly easier. If we write out the letters of $\delta$ and $\alpha_{m+1}\cdots\alpha_k$ as $\delta = d_1 d_2 \cdots d_p$ and $\alpha_{m+1}\cdots\alpha_k = a_1 a_2 \cdots a_q$, then we have 
\[
\tilde{\delta}=(d_1,d_p,\dots,d_2)\quad\text{and}\quad\tilde{\delta}^{\prime}=(x,\pi_{l-1},\dots,\pi_{j+1},\pi_j,d_p,\dots,d_2,d_1,a_q,\dots,a_2,a_1).
\]
From here, we see that $o(\Phi(\pi))$ and $o(\Phi(\sigma))$ are the same except for the positions of the letters $d_1=\beta_1(\pi),\pi_j,\pi_{j+1},\dots,\pi_{l-1},x$. More precisely, to obtain $o(\Phi(\sigma))$ from $o(\Phi(\pi))$, we remove all of these letters from where they are initially located, insert $\beta_1(\pi)$ between $d_2$ and $a_q$, and prepend the remaining letters $\pi_j,\pi_{j+1},\dots,\pi_{l-1},x$ at the beginning but in reverse order. To complete the proof of this case, we shall argue that this arrangement is obtained precisely by applying $\varphi_S$ to $o(\Phi(\pi))$ for $S=\{\beta_1(\pi),\pi_j,\pi_{j+1},\dots,\pi_{l-1},x\}$.

From Lemma \ref{l-xbet}, we know that $\beta_1(\pi)$ is larger than all of the letters $d_2,\dots,d_p$, but $a_q>\beta_1(\pi)$ from the definition of rix-factorization. Hence, letting $\beta_1(\pi)$ hop in $o(\Phi(\pi))$ will move $\beta_1(\pi)$ to the desired position. Furthermore, by the fact that the rixed points of $\pi$ are all greater than $\beta_1(\pi)$ and by the definition of canonical cycle representation, each of the letters $\pi_j,\pi_{j+1},\dots,\pi_{l-1},x$ is larger than all letters to its left in $o(\Phi(\pi))$. And since $\pi_j<\pi_{j+1}<\cdots<\pi_{l-1}<x$, letting all of these letters hop will move them to the beginning in reverse order. Hence we have $o(\Phi(\sigma))=\varphi_S(o(\Phi(\pi)))$, and applying $o^{-1}$ to both sides gives us $\Phi(\sigma)=\psi_S(\Phi(\pi))$.

\medskip{}
\noindent \textbf{Case 5:} $x=\beta_1(\pi)$ is not a rixed point of $\pi$.

Write $\beta = x d_2 \cdots d_l \pi_j \cdots \pi_n$ where $\pi_j,\dots,\pi_n$ are the rixed points of $\pi$, so 
\[
\delta = x d_2 \cdots d_l \quad\text{and}\quad\tilde{\delta}=(x,d_l,d_{l-1}\dots,d_2).
\]
By Lemma \ref{l-xbet}, we know that $x$ is a double descent of $\pi$, that $d_2\cdots d_l$ is nonempty, and that $x$ is larger than all of the letters $d_2,\dots,d_l$. 
Thus, in applying $\varphi_x$ to $\pi$, the rix-factors $\alpha_i$ will be unchanged but $x$ will hop over all of the letters $d_2,\dots,d_l$. In other words, we have $\alpha_i^{\prime}=\alpha_i$ for all $1\leq i\leq k$ and 
\[
d_2 \cdots d_l x \pi_j \cdots \pi_n=\alpha_{k+1}^{\prime}\cdots\alpha_m^{\prime}\underset{\beta^{\prime}}{\underbrace{d_p d_{p+1} \cdots d_l x \pi_j \cdots \pi_n}}
\]
for some $2\leq p\leq l$. Note that $x$ is a rixed point of $\sigma$ because $x>d_p$ and is part of an increasing suffix of $\sigma$. Let us write $\Rix(\sigma)=\{d_q,d_{q+1}\dots,d_l,x,\pi_j,\dots,\pi_n\}$ so that 
\[
\beta^{\prime}=\underset{\delta^{\prime}}{\underbrace{d_p d_{p+1} \cdots d_{q-1}}} d_q d_{q+1} \cdots d_l x \pi_j \cdots \pi_n.
\]
Thus $\Phi(\sigma)$, in canonical cycle representation, begins with the cycles
\[
\underset{\tilde{\delta}^{\prime}}{\underbrace{(d_p,d_{q-1},\dots,d_{p+1})}}\tilde{\alpha}_m^{\prime}\cdots\tilde{\alpha}_{k+1}^{\prime}.
\]
Note that the letters in the cycles $\tilde{\alpha}_m^{\prime}\cdots\tilde{\alpha}_{k+1}^{\prime}$ are precisely $d_{p-1},\dots,d_3,d_2$ in this given order. Comparing with $\tilde{\delta}=(x,d_l,d_{l-1}\dots,d_2)$, we see that in going from $o(\Phi(\pi))$ to $o(\Phi(\sigma))$, the only difference is in the positions of the letters $d_p=\beta_1(\sigma)$ and the letters $x,d_q,d_{q+1},\dots,d_l$ (the rixed points of $\sigma$ smaller than or equal to $x$). We must show that the movement of these letters corresponds precisely to letting them hop---that is, applying $\varphi_S$ to $o(\Phi(\pi))$ where $S=\{d_p,x,d_q,d_{q+1},\dots,d_l\}$ results in $o(\Phi(\sigma))$.

Let $y$ be any of the letters $x,d_q,d_{q+1},\dots,d_l$, which are all rixed points of $\sigma$ and thus fixed points of $\Phi(\sigma)$. Per canonical cycle representation, to go from $o(\Phi(\pi))$ to $o(\Phi(\sigma))$, each of these $y$ must be moved to the position immediately before the closest letter to the right of $y$ that is larger than $y$, which is the first entry of the cycle immediately after $(y)$ in $\Phi(\sigma)$. Now, recall that
\[
d_q<d_{q+1}<\cdots<d_l<x\qquad\text{and}\qquad\tilde{\delta}=(x,d_l,d_{l-1}\dots,d_2);
\]
together, these imply that $d_q$ is either a double descent or valley\footnote{It is possible for $d_q$ to be a valley only when $d_q = d_p$, i.e, when $\delta^{\prime}$ is empty.} of $o(\Phi(\pi))$, and all the other $y$ are double descents of $o(\Phi(\pi))$, so they will all hop to the right (or remain stationary) in applying $\varphi_S$. By the definition of valley-hopping, each of these $y$ will move to the position immediately before the closest letter to the right of $y$ that is larger than $y$, precisely as described above. Hence, in applying $\varphi_S$ to $o(\Phi(\pi))$, all of the letters $x,d_q,d_{q+1},\dots,d_l$ will move to the desired positions. It remains to show that $d_p$ will move to the correct position as well.

At this point, we note that it is possible for $d_p=\beta_1(\sigma)$ to be a rixed point of $\sigma$, and in this case we would have $d_p=d_q$ and the proof would be complete. So let us assume that $d_p$ is not a rixed point of $\sigma$. In going from $o(\Phi(\pi))$ to $o(\Phi(\sigma))$, the letter $d_p$ is moved to the very beginning of the permutation. Upon letting the letters $x,d_q,d_{q+1},\dots,d_l$ hop in $o(\Phi(\pi))$, only the letters $d_{q-1},\dots,d_{p+1}$ appear before $d_p$, so it suffices to show that $d_p$ is a double ascent of $o(\Phi(\pi))$ and that it is larger than all of the letters $d_{q-1},\dots,d_{p+1}$. The latter follows from Lemma \ref{l-xbet}, as $d_p=\beta_1(\sigma)$ and $d_{q-1},\dots,d_{p+1}$ are precisely the letters in $\beta^{\prime}$ which are neither $\beta_1(\sigma)$ nor rixed points of $\sigma$. To see that $d_p$ is a double ascent of $o(\Phi(\pi))$, we first note that $d_p$ is preceded by $d_{p+1}$ in $o(\Phi(\pi))$ and $d_{p+1}<d_p$. Now we consider two subcases:
\begin{itemize}
\item If $p=2$, then $d_p$ is the last entry of the cycle $\tilde{\delta}$ of $\Phi(\pi)$, which by canonical cycle representation implies that $d_p$ is a double ascent of $o(\Phi(\pi))$.
\item If $p>2$, then $d_p$ is followed by $d_{p-1}$ in $o(\Phi(\pi))$. Note that $d_{p-1}$ is the first entry of the cycle $\tilde{\alpha}_m^{\prime}$ appearing after $\tilde{\delta}^{\prime}$ in $\Phi(\sigma)$, which by canonical cycle representation implies that $d_p<d_{p-1}$. Hence, $d_p$ is a double ascent of $o(\Phi(\pi))$.
\end{itemize}
Therefore, the remaining letter $d_p$ also moves to the correct position after applying $\varphi_S$, and the proof is complete.
\end{proof}

\begin{cor} \label{c-oneinc}
Let $\pi$ be a permutation, $\Pi$ the valley-hopping orbit containing $\pi$, $\hat{\Pi}$ the restricted valley-hopping orbit containing $\pi$, $\Pi^{\prime}$ the cyclic valley-hopping orbit containing $\Phi(\pi)$, and $\hat{\Pi}^{\prime}$ the restricted cyclic valley-hopping orbit containing $\Phi(\pi)$. Then:
\begin{enumerate}
\item[\normalfont(a)]  $\Phi(\Pi)\subseteq\Pi^{\prime}$
\item[\normalfont(b)]  $\Phi(\hat{\Pi})\subseteq\hat{\Pi}^{\prime}$
\end{enumerate}
\end{cor}

\begin{proof}
The proof of part (a) from Lemma \ref{l-Phiclosed} is straightforward and so it is omitted. Here we prove (b) as it requires a more subtle argument.

Let $\sigma\in\hat{\Pi}$. Then, there exists $S=\{x_{1},x_{2},\dots,x_{k}\}$ for which
\[
\sigma=\hat{\varphi}_{S}(\pi)=(\hat{\varphi}_{x_{k}}\circ\cdots\circ\hat{\varphi}_{x_{2}}\circ\hat{\varphi}_{x_{1}})(\pi).
\]
Assume without loss of generality that none of the $x_{i}$ is $\beta_{1}(\pi)$ or a rixed point of $\pi$. According to \cite[Lemma 13]{Lin2015}, $\beta_{1}(\pi)$ and $\Rix(\pi)$ are invariant under restricted valley-hopping, which means that none of the $x_{i}$ is the first letter of the $\beta$ rix-factor or is a rixed point of any of the permutations
\[
\hat{\varphi}_{x_{1}}(\pi),\;(\hat{\varphi}_{x_{2}}\circ\hat{\varphi}_{x_{1}})(\pi),\;\dots,\;(\hat{\varphi}_{x_{k}}\circ\cdots\circ\hat{\varphi}_{x_{2}}\circ\hat{\varphi}_{x_{1}})(\pi)
\]
in the same restricted valley-hopping orbit as $\pi$. Thus we have $\hat{\varphi}_{x_{i}}=\varphi_{x_{i}}$ for each $1\leq i\leq k$; that is,
\[
\sigma=\varphi_{S}(\pi)=(\varphi_{x_{k}}\circ\cdots\circ\varphi_{x_{2}}\circ\varphi_{x_{1}})(\pi).
\]
Applying Lemma \ref{l-Phiclosed} (a), we deduce
\[
\Phi(\sigma)=\psi_{S}(\Phi(\pi))=(\psi_{x_{k}}\circ\cdots\circ\psi_{x_{2}}\circ\psi_{x_{1}})(\Phi(\pi)).
\]
Observe from the definition of $\Phi$ that the rixed points of a permutation $\tau$ are precisely the fixed points of $\Phi(\tau)$, and that $\beta_{1}(\tau)$ is the first letter of $o(\Phi(\tau))$. Consequently, we have that $\psi_{x_{i}}=\hat{\psi}_{x_{i}}$ for each $1\leq i\leq k$, which implies $\Phi(\sigma)=\hat{\psi}_{S}(\Phi(\pi))$. Hence, (b) is proven.
\end{proof}

Theorem \ref{t-main} now follows easily from Corollaries \ref{c-orbsize} and \ref{c-oneinc}.

\begin{proof}[of Theorem \ref{t-main}]
By Corollary \ref{c-orbsize} and the fact that $\Phi$ is a bijection, we have
\[
|\Phi(\Pi)|=|\Pi|=|\Pi^{\prime}|\qquad\text{and}\qquad|\Phi(\hat{\Pi})|=|\hat{\Pi}|=|\hat{\Pi}^{\prime}|,
\]
which together with $\Phi(\Pi)\subseteq\Pi^{\prime}$ and $\Phi(\hat{\Pi})\subseteq\hat{\Pi}^{\prime}$ (Corollary \ref{c-oneinc}) yield the desired results.
\end{proof}

\acknowledgements
We thank Tom Roby for helpful discussions on this project.

\bibliographystyle{plain}
\bibliography{bibliography}

\end{document}